\documentclass[12pt]{article}

\usepackage{amsfonts}
\usepackage{amssymb}
\usepackage{amsthm}
\usepackage{graphicx}
\usepackage[lofdepth,lotdepth]{subfig}
\usepackage{times}
\usepackage{rotating}

\newtheorem{thm}{Theorem}[section]
\newtheorem{prop}[thm]{Proposition}

\newtheorem{lem}[thm]{Lemma}

\theoremstyle{definition}
\newtheorem{example}[thm]{Example}
\newtheorem{definition}[thm]{Definition}

\newtheorem*{ack}{Acknowledgements}

\newcommand{\im}{\textup{im }}

\newcommand{\Z}{\mathbb Z}

\newcommand{\Comment}[1]{}

\newcommand{\commentout}[1]{{}}

\newcommand{\idiot}[1]{\vspace{5 mm}\par \noindent
\framebox{\begin{minipage}[c]{0.95 \textwidth}
\tt #1 \end{minipage}}\vspace{5 mm}\par}

\begin{document}

\title{On $d$-dimensional cycles and the vanishing of simplicial homology}

\author{ E. Connon\thanks{Research supported by a Killam scholarship.} }

\maketitle

\begin{center} \it \small
Department of Mathematics and Statistics, Dalhousie University, Canada
\end{center}

\begin{abstract}
In this paper we introduce the notion of a $d$-dimensional cycle which is a homological generalization of the idea of a graph cycle to higher dimensions. We examine both the combinatorial and homological properties of this structure and use these results to describe the relationship between the combinatorial structure of a simplicial complex and its simplicial homology.  In particular, we show that over any field of characteristic 2 the existence of non-zero $d$-dimensional homology corresponds exactly to the presence of a $d$-dimensional cycle in the simplicial complex.  We also show that $d$-dimensional cycles which are orientable give rise to non-zero simplicical homology over any field.
\end{abstract}



\section{Introduction} \label{sec:intro}

In combinatorial commutative algebra one of the goals is to find descriptions of algebraic properties of monomial ideals through the characteristics of associated combinatorial structures such as the Stanely-Reisner complex or the facet complex.  In the case that the facet complex is 1-dimensional we can look at the monomial ideal as the edge ideal of a graph.  In this case significant relationships have been found between the combinatorial structure of the graph and the algebraic characteristics of the ideal.  For example, it is shown in \cite{Fr90} that the edge ideal of a graph has a linear resolution if and only if the complement of the graph is chordal.  In \cite{HHZ06} it is determined that the edge ideal of a chordal graph is Cohen-Macaulay if and only if it is unmixed.  In higher dimensions these types of results are often not complete classifications.  Rather, there are many interesting theorems showing that certain classes of simplicial complexes or hypergraphs have associated ideals with certain algebraic properties.  For several examples see \cite{Em10}, \cite{Far04}, \cite{HaVT08}, \cite{MYZ12}, and \cite{Wood11}.

Often, however, there exist complete characterizations of algebraic features of square-free monomial ideals that are generated in any degree in terms of the simplicial homology of their associated simplicial complexes.  For example, a theorem of Reisner from \cite{Reis76} gives a condition for the Stanley-Reisner ring of a simplicial complex to be Cohen-Macaulay in terms of the simplicial homology of the complex and some of its subcomplexes.  In \cite{Fr85} Fr\"oberg classifies the monomial ideals with linear resolutions through the simplicial homology of their Stanley-Reisner complexes and their induced subcomplexes.  Nevertheless a simplicial complex may be thought of as a purely combinatorial object and so the question arises as to whether or not one may attribute the existence of non-zero homology in a simplicial complex to a combinatorial structure present in that complex.  This is an interesting mathematical question in its own right, but at the same time its possible solution has the potential to enable easier translation between algebraic properties of monomial ideals and the combinatorial framework.

Despite the substantial use of simplicial homology in classifying algebraic properties and the concrete combinatorial nature of the theory there appears to be very little literature on the explicit combinatorial structures necessary for a simplicial complex to exhibit non-zero simplicial homology.  However, studies have been made into the combinatorics of acyclic simplicial complexes -- those for which simplicial homology vanishes in all dimensions.  In \cite{Kal84} Kalai gave a characterization of the $f$-vectors of such simplicial complexes.  This was followed by work of Stanley in \cite{Stan93} on a combinatorial decomposition of these complexes.  On the other hand, the literature relating to the combinatorial structures associated with non-zero simplicial homology is scant.  In \cite{Fog88} Fogelsanger studied the rigidity of the $1$-skeletons of the support complexes of minimal $d$-cycles and concluded that these $1$-skeletons are rigid graphs when embedded in $\mathbb{R}^d$.   However, the question remains as to whether or not one may describe non-zero homology in a simplicial complex using purely combinatorial properties.

In this paper we introduce the notion of a {\bf $d$-dimensional cycle}.  This is a homological generalization to higher dimensions of the idea of a graph cycle.  We examine the structure of these higher-dimensional cycles and prove results about their combinatorial properties.  We will then show that, over fields of characteristic $2$, non-zero $d$-dimensional homology can be attributed to the presence of $d$-dimensional cycles.  We also show that {\bf orientable} $d$-dimensional cycles lead to non-zero homology over fields of any characteristic.

In \cite{ConFar12} and \cite{ConFar13} the results of this paper are used to study the combinatorial structure of simplicial complexes whose Stanley-Reisner ideals have linear resolutions.

\begin{ack}
The author is grateful to Manoj Kummini for his comments and would also like to thank her supervisor, Sara Faridi, for help in the preparation of this paper.
\end{ack}

\section{Preliminaries} \label{sec:prelim}
\subsection{Simplicial complexes and simplicial homology} \label{sec:simp_cmpx_hom}

An (abstract) {\bf simplicial complex} $\Delta$ is a family of subsets of a finite vertex set $V(\Delta)$ that is closed under taking subsets.  In other words, if $S \in \Delta$ and $S' \subseteq S$ then $S' \in \Delta$. These subsets are called {\bf faces} of $\Delta$ and the elements of $V(\Delta)$ are the {\bf vertices} of $\Delta$.  Those faces which are maximal with respect to inclusion are referred to as {\bf facets} of $\Delta$ and if the facets of $\Delta$ are $F_1,\ldots,F_k$ then we write
\[
    \Delta = \langle F_1,\ldots,F_k \rangle.
\]

The {\bf dimension} of a face $S$ of $\Delta$ is defined to be $|S|-1$ and it is denoted $\dim S$. A face of dimension $d$ is referred to as a {\bf $d$-face}. The {\bf dimension} of the simplicial complex $\Delta$, $\dim \Delta$, is the maximum dimension of any of its faces.  We say that $\Delta$ is {\bf pure} when all of its facets have the same dimension.  If $\Delta$ contains all possible faces of dimension $d$ then $\Delta$ is {\bf $d$-complete.} A {\bf $d$-simplex} is a simplicial complex with a single facet of dimension $d$.  A {\bf subcomplex} of the simplicial complex $\Delta$ is a simplicial complex whose faces are a subset of the faces of $\Delta$.  \Comment{For any subset $W$ of $V(\Delta)$, the {\bf induced subcomplex of $\Delta$ on $W$}, denoted $\Delta_W$, is the subcomplex of $\Delta$ whose vertex set is $W$ and whose faces are exactly those faces of $\Delta$ entirely contained in $W$.}

\Comment{Let $\Delta$ and $\Gamma$ be simplicial complexes with $V(\Delta) \cap V(\Gamma) = \emptyset$.  The {\bf join} of $\Delta$ and $\Gamma$ is the simplicial complex
\[
    \Delta \ast \Gamma = \{F \cup G \ | \ F \in \Delta, G \in \Gamma \}.
\]
It is not difficult to see that $\dim (\Delta \ast \Gamma) = \dim \Delta + \dim \Gamma + 1$.

\idiot{add picture of join}}

An ordering of the vertices in a face $S$ of a simplicial complex is called an {\bf orientation} of $S$ and two orientations are {\bf equivalent} when one is an even permutation of the other.  A $d$-face with a choice of one of these orientations is referred to as an {\bf oriented $d$-face}.  We denote by $[v_0,\ldots,v_d]$ the equivalence class of the $d$-face having vertices $v_0,\ldots,v_d$ and with the orientation $v_0 < \cdots < v_d$.

Letting $A$ be any commutative ring with identity we define $C_d(\Delta)$ to be the free $A$-module whose basis is the oriented $d$-faces of $\Delta$ with the relations $[v_0,v_1,\ldots,v_d]=-[v_1,v_0,\ldots,v_d]$.  We call the elements of $C_d(\Delta)$ {\bf $d$-chains} and we call the simplicial complex whose facets are the $d$-faces in a $d$-chain having non-zero coefficients the {\bf support complex} of the $d$-chain.

The natural boundary map homomorphism $\partial_d: C_d(\Delta) \rightarrow C_{d-1}(\Delta)$ is defined by
\[
    \partial_d([v_0,\ldots,v_d]) = \sum_{i=0}^{d}(-1)^i [v_0,\ldots,v_{i-1},v_{i+1},\ldots,v_d]
\]
for each oriented $d$-face $[v_0,\ldots,v_d]$.  We call the elements of the kernel of $\partial_d$ the {\bf $d$-cycles} and the elements of the image of $\partial_d$ the {\bf $(d-1)$-boundaries}.  We define the {\bf $d$th simplicial homology group} of $\Delta$ over $A$, $H_d(\Delta;A)$, to be the quotient of the group of $d$-cycles over the group of $d$-boundaries.

We can further define a surjective homomorphism
$\epsilon: C_0(\Delta) \rightarrow A$
by $\epsilon(v)=1$ for all $v \in V(\Delta)$. We can then define the {\bf reduced homology group of $\Delta$ in dimension $0$} as
\[
    \tilde{H}_0(\Delta;A) = \ker \epsilon / \im \partial_1.
\]
Setting $\tilde{H}_i(\Delta;A)=H_i(\Delta;A)$ for $i>0$ we obtain the {\bf $i$th reduced homology group} of $\Delta$.

The $d$th simplicial homology group of a simplicial complex $\Delta$ measures the structure of the ``holes'' of dimension $d$ in $\Delta$.  Our aim in this paper is to give concrete descriptions of these holes.  Specifically, we examine the support complexes of $d$-cycles from a combinatorial point of view.

For a more in-depth explanation of simplicial homology we refer the reader to \cite{Munk84}.

\subsection{Graph Theory}
A {\bf finite simple graph} $G$ consists of a finite set of {\bf vertices} denoted $V(G)$ and a set of {\bf edges} denoted $E(G)$ where the elements of $E(G)$ are unordered pairs from $V(G)$.  In this paper we will use the term {\bf graph} to refer to a finite simple graph.  A $1$-dimensional simplicial complex can be thought of as a graph where the edges are given by the $1$-faces of the simplicial complex.

A vertex $v$ of a graph $G$ is said to be {\bf incident} with an edge $e$ of $G$ if $v$ is a vertex of $e$.  The {\bf degree} of a vertex $v$ is the number of edges with which $v$ is incident.

A {\bf path} in a graph $G$ is a sequence of vertices $v_0,\ldots,v_n$ from $V(G)$ such that $\{v_{i-1},v_i\} \in E(G)$ for $1 \leq i \leq n$.  A {\bf connected} graph is a graph in which every pair of vertices is joined by a path.

A {\bf cycle} in a graph $G$ is an ordered list of distinct vertices $v_0,\ldots,v_n$ from $V(G)$ where $\{v_{i-1},v_i\} \in E(G)$ for $1 \leq i \leq n$ and $\{v_n,v_0\} \in E(G)$.  In Figure \ref{fig:six_cycle} we give an example of a graph cycle on six vertices.

\begin{figure}[h!]
{\centering
    \includegraphics[height=1.4in]{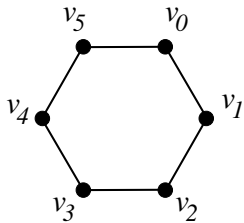}
    \caption {Graph cycle on six vertices.} \label{fig:six_cycle}
}
\end{figure}


\section{Motivation} \label{sec:hom_motivation}

Many of the homological classifications used in commutative algebra are based on the vanishing of simplicial homology in certain simplicial complexes.  We would like to turn this property into an explicitly combinatorial one and so our goal is to relate the concept of non-zero homology in a simplicial complex to the existence of a specific combinatorial structure in that complex which is distinct from the more algebraic notion of a $d$-chain.

It is not difficult to show that this goal is achievable for non-zero simplicial homology in dimension 1.  In this case, the combinatorial structure associated to non-zero homology is the graph cycle.  This can be deduced from \cite[Chapters 4 and 5]{Biggs93}, but in this section we provide an explicit proof of this relationship as motivation for the general case.  The techniques used in this proof also illustrate our overall approach to this problem.

First we recall the following well-known lemma from graph theory.  A proof can be found in \cite[Section 1.2]{West96}.

\begin{lem} \label{lem:contains_cycle}
If $G$ is a finite simple graph in which every vertex has degree at least $2$ then $G$ contains a cycle.
\end{lem}

\begin{thm}[{\bf Non-zero $1$-dimensional homology corresponds to graph cycles}] \label{thm:1hom_graph_cycle}
For any simplicial complex $\Delta$ and any field $k$, $\tilde{H}_1(\Delta;k) \neq 0$ if and only if $\Delta$ contains a graph cycle, which is not the support complex of a $1$-boundary.
\end{thm}

\begin{proof}
Suppose that $\tilde{H}_1(\Delta;k) \neq 0$.  Then $\Delta$ contains a $1$-cycle $c$ that is not a $1$-boundary.  We may assume that the support complex $\Omega$ of $c$ is minimal with respect to this property.  In other words no strict subset of the $1$-faces of $c$ is the support complex of a $1$-cycle which is not a $1$-boundary.  First we would like to show that $\Omega$ is a connected graph.

The $1$-cycle $c$ is of the form
\begin{equation}\label{eq:c}
    c=\alpha_1 F_1+\cdots +\alpha_n F_n
\end{equation}
for some oriented $1$-faces $F_1,\ldots,F_n$ of $\Delta$ and where $\alpha_i \in k$.  If $\Omega$ is not connected then we can partition the $1$-faces $F_1,\ldots,F_n$ into two sets having no vertices in common.  Without loss of generality let these two sets be $\{F_1,\ldots,F_\ell\}$ and $\{F_{\ell+1},\ldots,F_n\}$.  Since
\[
    \partial_1(c)=\partial_1(\alpha_1 F_1+\cdots +\alpha_n F_n)=0
\]
and since there are no vertices shared between the two sets we must have
\[
    \partial_1(\alpha_1 F_1+\cdots +\alpha_\ell F_\ell)=0 \; \textrm{ and } \; \partial_1(\alpha_{\ell+1} F_{\ell+1}+\cdots +\alpha_n F_n)=0
\]
and so $\alpha_1 F_1+\cdots +\alpha_\ell F_\ell$ and $\alpha_{\ell+1} F_{\ell+1}+\cdots +\alpha_n F_n$ are both $1$-cycles.  By the assumption of minimality of $\Omega$ we know that these $1$-cycles must also be $1$-boundaries.  However, since $c$ is the sum of these two $1$-chains and they are both $1$-boundaries, $c$ must be a $1$-boundary as well, which is a contradiction.  Therefore $\Omega$ must be a connected graph.

Next, note that the degree of all vertices in $\Omega$ must be at least two.  This follows since $\partial_1(c)=0$ and this may only be achieved if all vertices present cancel out in this sum.  Therefore each vertex must appear at least twice.  Hence by Lemma \ref{lem:contains_cycle} we know that $\Omega$ contains a graph cycle.  Let $v_0,\ldots,v_m$ be the vertices in this cycle where $v_i$ is adjacent to $v_{i+1}$ for $0 \leq i \leq m-1$ and $v_m$ is adjacent to $v_0$.  By relabeling if necessary we may assume that $F_1,\ldots,F_m$ are the oriented $1$-faces corresponding to the edges in this cycle with $F_i = \varepsilon_i [v_i,v_{i+1}]$ for $0 \leq i \leq m-1$ and $F_m = \varepsilon_m [v_m,v_0]$ where $\varepsilon_i = \pm 1$ depending on the orientation of $F_i$.  Suppose that $m < n$.    Then
\[
    \varepsilon_1 F_1+\ldots +\varepsilon_m F_m=\sum_{i=0}^{m-1}[v_i,v_{i+1}] + [v_m,v_0]
\]
is a $1$-chain and it is straightforward to see that
\[
    \partial_1(\varepsilon_1 F_1+\ldots +\varepsilon_m F_m)=0.
\]
Let $b= \alpha_1\varepsilon_1 (\varepsilon_1 F_1+\cdots + \varepsilon_m F_m)$.  Then
\[
    \partial_1(b)=\alpha_1\varepsilon_1 \partial_1(\varepsilon_1 F_1+\ldots +\varepsilon_m F_m) = 0
\]
and so $b$ is a $1$-cycle.  We also have
\[
    \partial_1(c-b)=\partial_1(c)-\partial_1(b) =0-0=0
\]
and so $c-b$ is also a $1$-cycle.  Now $m<n$ and so by our assumption of minimality of $\Omega$ we know that $b$ is a $1$-boundary.  Since
\[
    c-b = (\alpha_2-\alpha_1\varepsilon_1\varepsilon_2)F_2 +(\alpha_3-\alpha_1\varepsilon_1\varepsilon_3)F_3+\cdots (\alpha_m-\alpha_1\varepsilon_1\varepsilon_m)F_m +\alpha_{m+1}F_{m+1}+\cdots+\alpha_nF_n,
\]
it is supported on a strict subset of the $1$-faces of $\Omega$ and so, by the assumption of minimality, $c-b$ is a $1$-boundary also.  Therefore, since both $b$ and $c-b$ are $1$-boundaries and $c=b+(c-b)$ then $c$ is a $1$-boundary.  This is a contradiction and so we must have $m=n$.  Therefore $\Omega$ itself is a graph cycle.

Next we will show that $\Omega$ is not the support complex of a $1$-boundary.  First we show that for any $1$-cycle of the form
\[
    d=\beta_1F_1+\cdots+\beta_nF_n
\]
for non-zero $\beta_1,\ldots,\beta_n \in k$ we have $\beta_1=\cdots=\beta_n$.  Let $v_1^i,v_2^i$ be the two vertices belonging to the $1$-face $F_i$.  Since $d$ is a $1$-cycle we have
\begin{equation}\label{eq:graph_cycle}
    0=\partial_1(d)=\beta_1v_1^1-\beta_1v_2^1+\cdots+\beta_n v_1^n -\beta_n v_2^n
\end{equation}
and, since $\Omega$ is a graph cycle, each $v_s^t$ appears exactly twice in (\ref{eq:graph_cycle}) as it belongs to two $1$-faces.  Therefore if $F_i$ and $F_j$ are any two $1$-faces which share a vertex then we get $\beta_i=\beta_j$ from (\ref{eq:graph_cycle}) and so any two $1$-faces sharing a vertex have the same coefficient.  Since $\Omega$ is connected we have $\beta_1=\cdots=\beta_n$.

Therefore $\alpha_1=\cdots=\alpha_n$ in (\ref{eq:c}) and if there exists a $1$-cycle in $\Delta$ with $\Omega$ as a support complex which is also a $1$-boundary then this implies that $c$ is also a $1$-boundary after multiplication by a constant in $k$.  This is a contradiction and so $\Omega$ is a graph cycle which is not the support complex of a $1$-boundary.

Conversely, suppose that $\Delta$ contains a graph cycle $\Omega$, which is not the support complex of a $1$-boundary.  Let $v_0,\ldots,v_k$ be the vertices of $\Omega$ where $v_i$ is adjacent to $v_{i+1}$ for $0 \leq i \leq k-1$ and $v_k$ is adjacent to $v_0$.  Then
\[
    c=\sum_{i=0}^{k-1}[v_i,v_{i+1}] + [v_k,v_0]
\]
is a $1$-chain whose support complex is $\Omega$ and it is easy to see that $\partial_1(c)=0$.  Therefore $c$ is a $1$-cycle, and by assumption it is not a $1$-boundary.  Therefore $\tilde{H}_1(\Delta;k) \neq 0$.

\end{proof}


\section{$d$-Dimensional Cycles} \label{sec:ddimcycles}
The role of the graph cycle in graph theory is substantial.  The existence, frequency, and lengths of cycles in a graph play important roles in many different areas of graph theory such as connectivity, perfect graphs, graph colouring, networks, and extremal graph theory.  Higher-dimensional versions of the graph cycle such as the Berge cycle (Berge \cite{Berge89}) and the simplicial cycle (Caboara et al. \cite{CFS07}) have been introduced to extend the usefulness of this concept to hypergraphs or simplicial complexes.

In recent years, with the introduction of edge ideals of graphs by Villarreal in \cite{Vill90}, the influence of the graph cycle has extended into combinatorial commutative algebra.  As we saw in Section \ref{sec:hom_motivation}, the graph cycle, when thought of as a simplicial complex, is exactly the right structure to describe non-zero simplicial homology in dimension one.  We would like to extend the idea of a graph cycle to higher dimensions in order to capture the idea of non-zero higher-dimensional simplicial homology.

When we examine the support complexes of ``minimal'' homological $d$-cycles we see that such complexes must be connected in a particularly strong way.  We will use the following definitions.

\begin{definition} [{\bf $d$-path, $d$-path-connected, $d$-path-connected components}]
A sequence $F_1,\ldots,F_k$ of $d$-dimensional faces in a simplicial complex $\Delta$ is a {\bf $d$-path} between $F_1$ and $F_k$ when $|F_i \cap F_{i+1}| = d$ for all $1 \leq i \leq k-1$.  If $\Delta$ is a pure $d$-dimensional simplicial complex and there exists a $d$-path between each pair of its $d$-faces then $\Delta$ is {\bf $d$-path-connected}.  The maximal subcomplexes of $\Delta$ which are $d$-path-connected are called the {\bf $d$-path-connected components} of $\Delta$.
\end{definition}

Note that a $d$-path-connected simplicial complex is sometimes referred to as {\bf strongly connected} (see, for example, \cite{Bjorn96}). In Figure \ref{fig:dpath} we give an example of a $2$-path between the $2$-faces $F_1$ and $F_2$.  Figure \ref{fig:dpathcomponents} shows a pure $2$-dimensional simplicial complex with two $2$-path-connected components shown by different levels of shading.

\Comment{\begin{figure}[h!]
{\centering
    \includegraphics[height=1in]{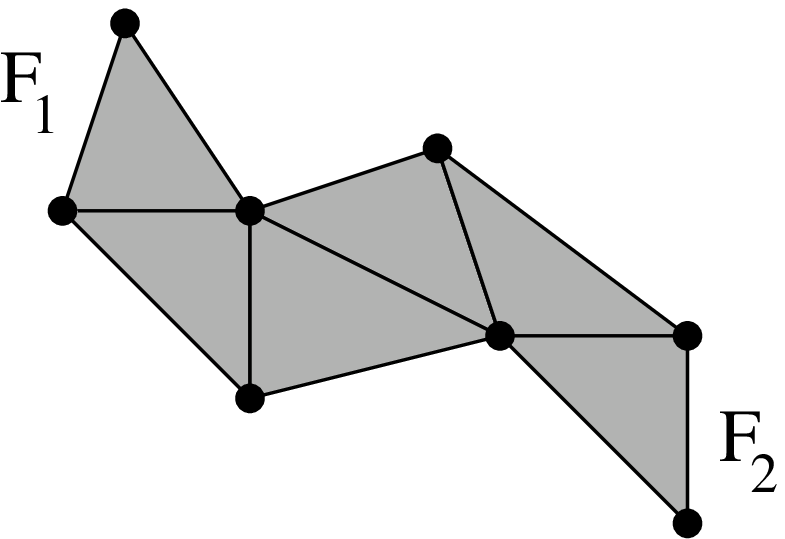}
    \caption {A $2$-path between $F_1$ and $F_2$.} \label{fig:dpath}

}
\end{figure}}

\begin{figure}[h]
\centering
\subfloat[A $2$-path between $F_1$ and $F_2$]{\makebox[7cm]{
            \includegraphics[height=1.2in]{dpath}}
            \label{fig:dpath}} \qquad
\subfloat[$2$-path-connected components]{\makebox[7cm]{
	\includegraphics[height=.9in]{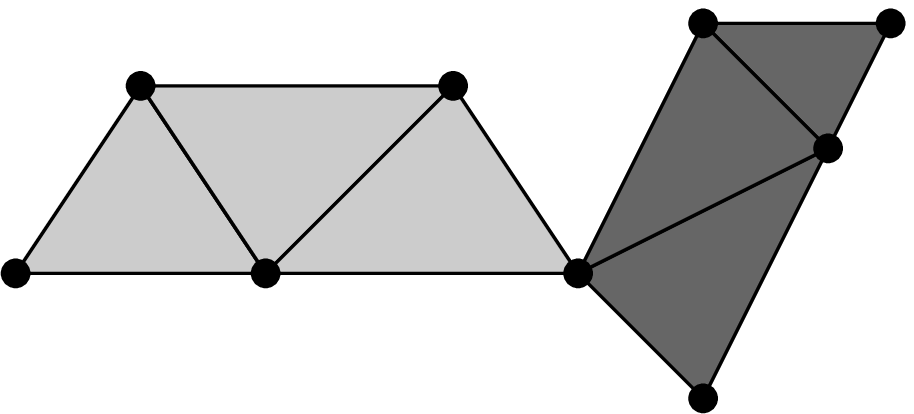}}
	\label{fig:dpathcomponents}} \\
\caption{Examples for path-connected complexes.} \label{fig:2d_cycles}
\end{figure}

A graph cycle is characterized by two features.  It is connected and each of its vertices is of degree two.  By generalizing these two properties we arrive at our combinatorial definition of a higher-dimensional cycle.

\begin{definition} [{\bf $d$-dimensional cycle}]
If $\Omega$ is a pure $d$-dimensional simplicial complex such that
\begin{enumerate}
\item $\Omega$ is $d$-path-connected, and
\item every $(d-1)$-face of $\Omega$ belongs to an even number of $d$-faces of $\Omega$
\end{enumerate}
then $\Omega$ is a {\bf $d$-dimensional cycle}.
\end{definition}

Notice that, by definition, a $d$-dimensional cycle has only one $d$-path-connected component.  We will see in Proposition \ref{prop:smallest_cycle} that a $d$-dimensional cycle must contain at least $d+2$ facets.

Several examples of $2$-dimensional cycles are given in Figure \ref{fig:2d_cycles}.

\begin{figure}[h]
\centering
\subfloat[A hollow octahedron.]{\makebox[7cm]{
            \includegraphics[height=1.4in]{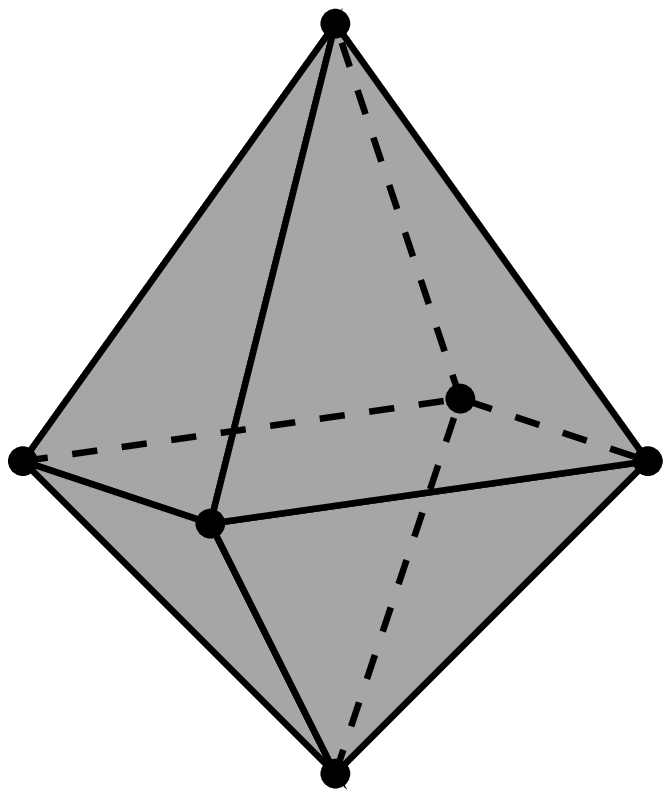}}
            \label{fig:octacycle}} \qquad
\subfloat[A triangulation of the sphere.]{\makebox[7cm]{
	\includegraphics[height=1.2in]{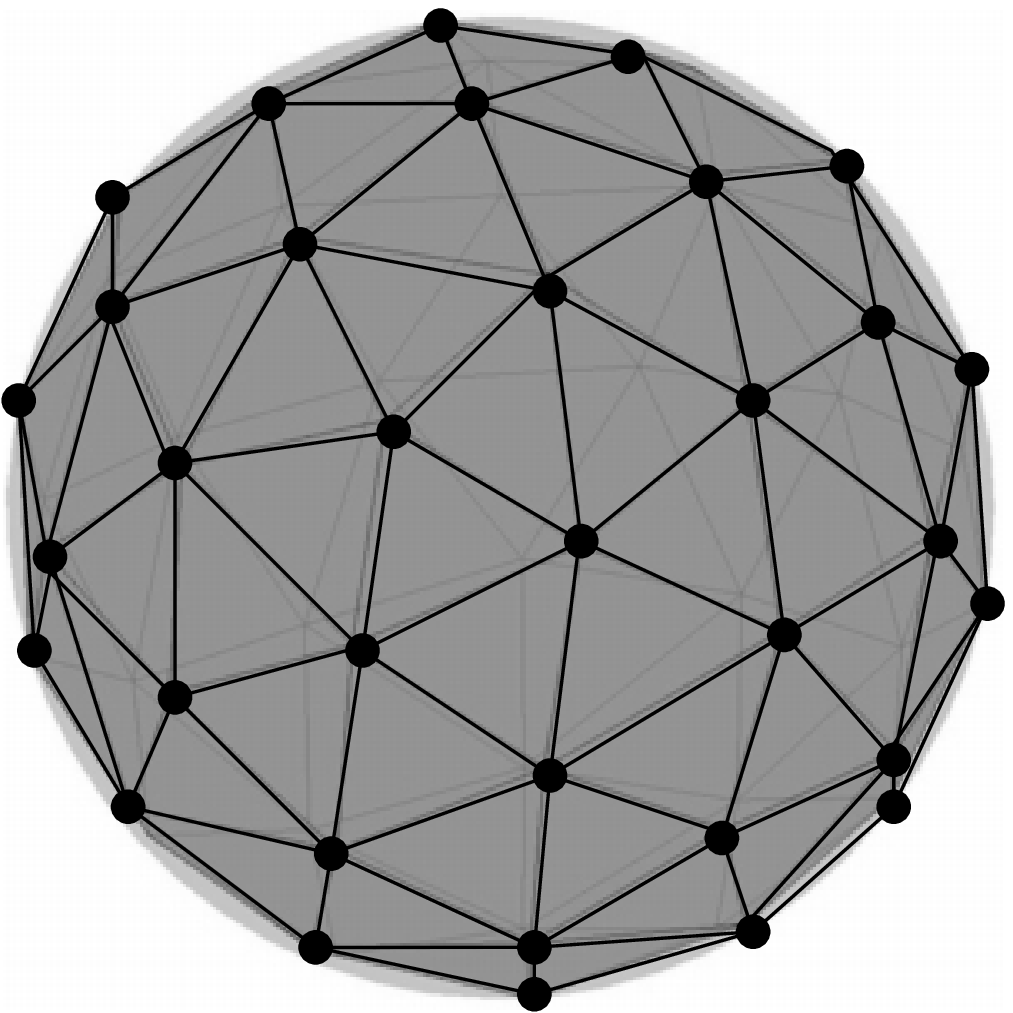}}
	\label{fig:sphere}} \\
\subfloat[A triangulation of the torus.]{\makebox[7cm]{
	\includegraphics[height=1.3in]{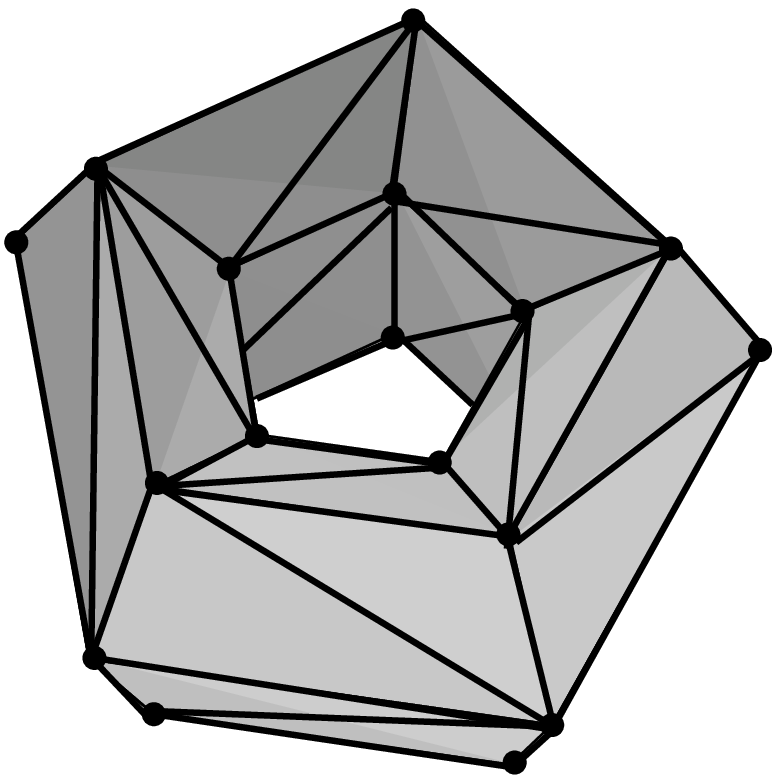}}
	\label{fig:torus}} \qquad
\subfloat[A triangulation of the real projective plane.]{\makebox[7cm]{
	\includegraphics[height=1.2in]{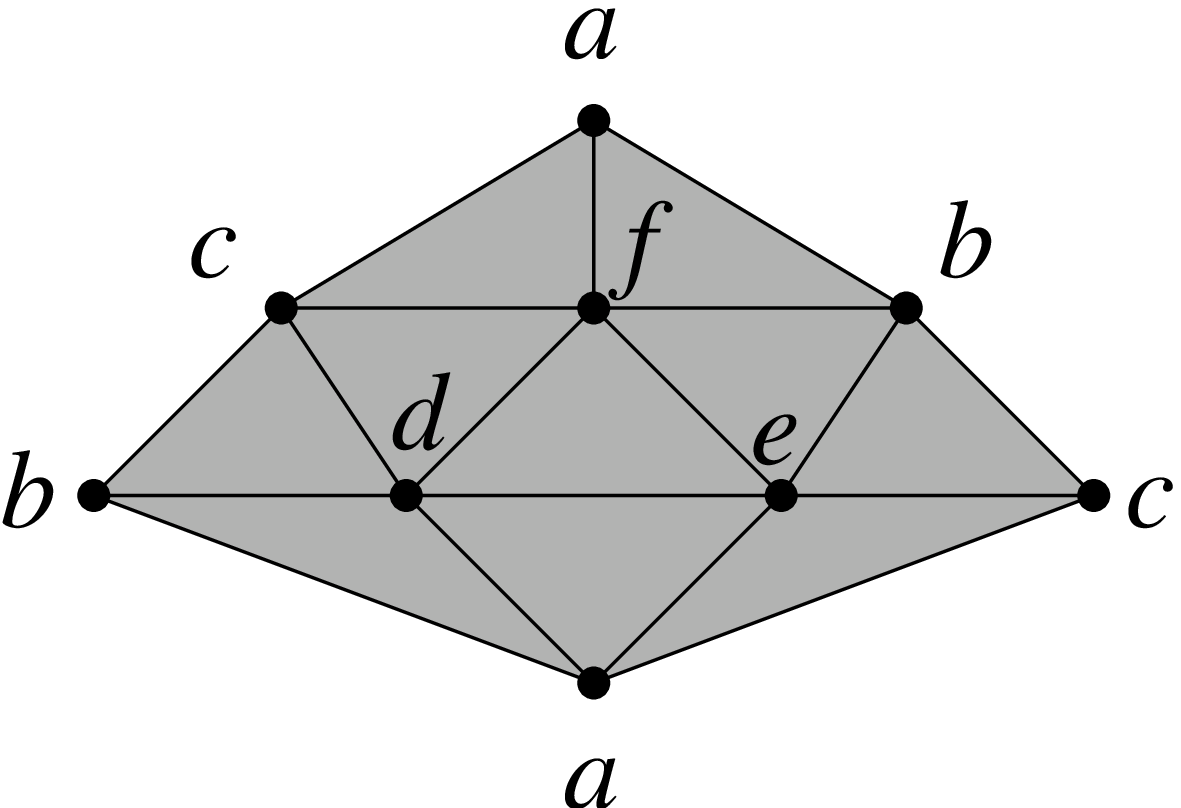}}
	\label{fig:real_proj_plane}} \\
\subfloat[Two triangulated square pyramids glued along a $1$-face.]{\makebox[7cm]{
	\includegraphics[height=1.1in]{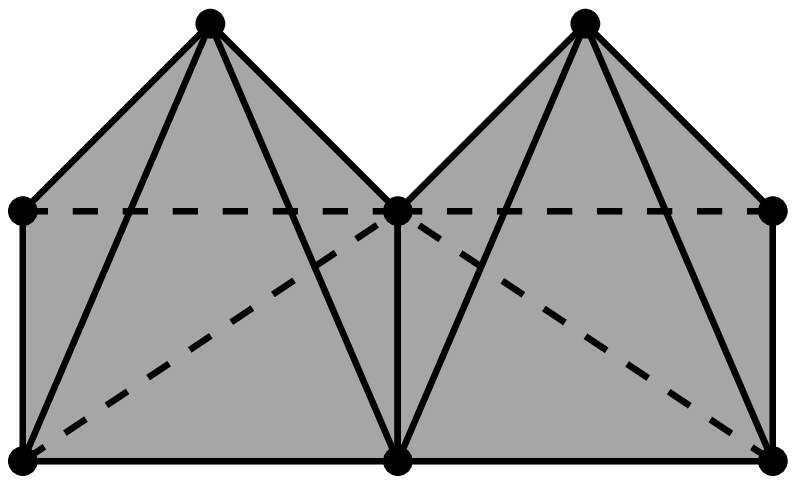}}
	\label{fig:two_pyramid}}\qquad
\subfloat[A triangulation of the sphere pinched along a $1$-face.]{\makebox[7cm]{
	\includegraphics[height=1.3in]{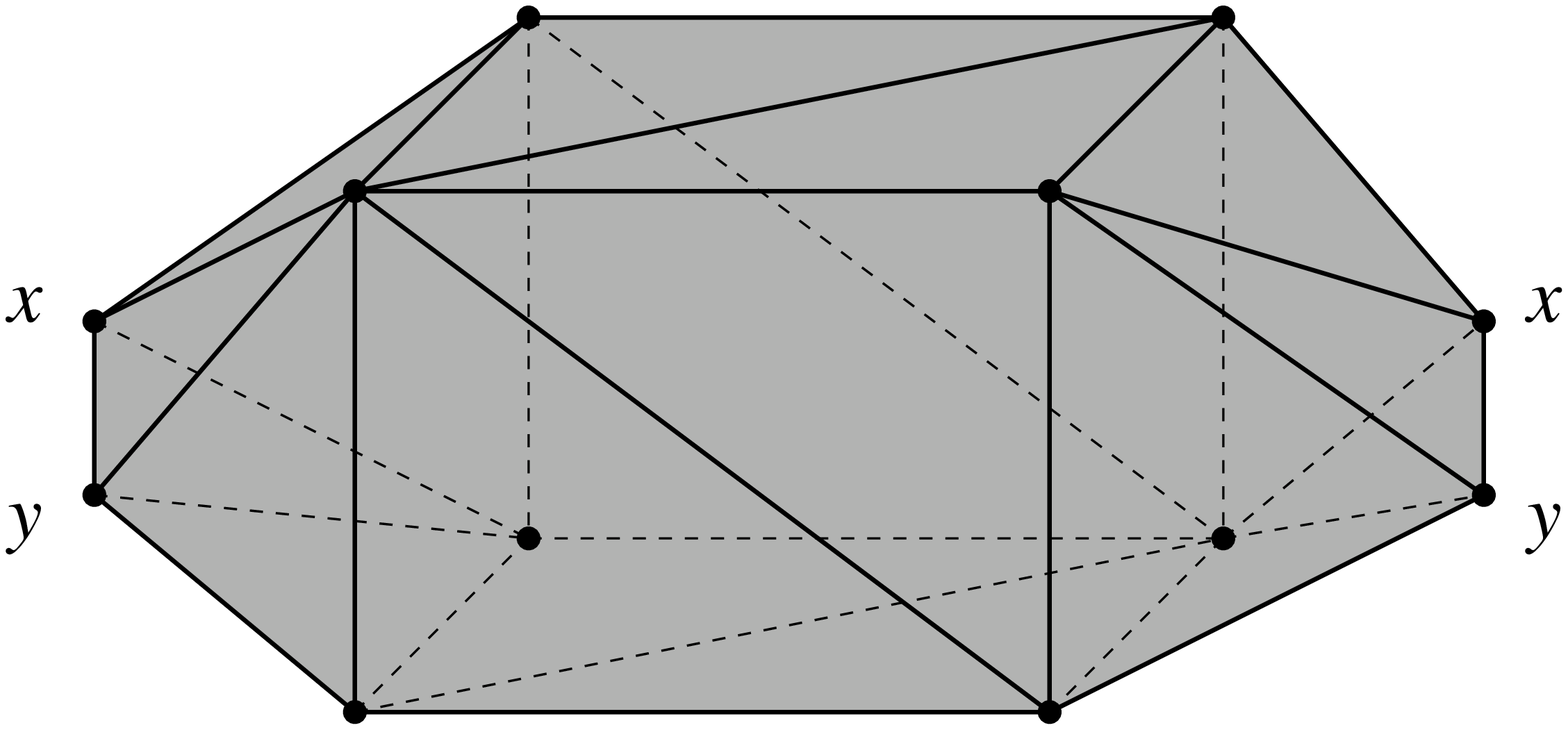}}
	\label{fig:pinched_sphere}}
\caption{Examples of $2$-dimensional cycles.} \label{fig:2d_cycles}
\end{figure}

One of the properties of a $d$-dimensional cycle is that it contains $(d-1)$-dimensional cycles as subcomplexes.  As an example, in Figure \ref{fig:cycle_within_cycle} we have a $2$-dimensional cycle containing a $1$-dimensional cycle within the $2$-faces containing the vertex $v$.  The $1$-dimensional cycle is shown with dotted lines.

\begin{figure}[h!]
{\centering
    \includegraphics[height=1.4in]{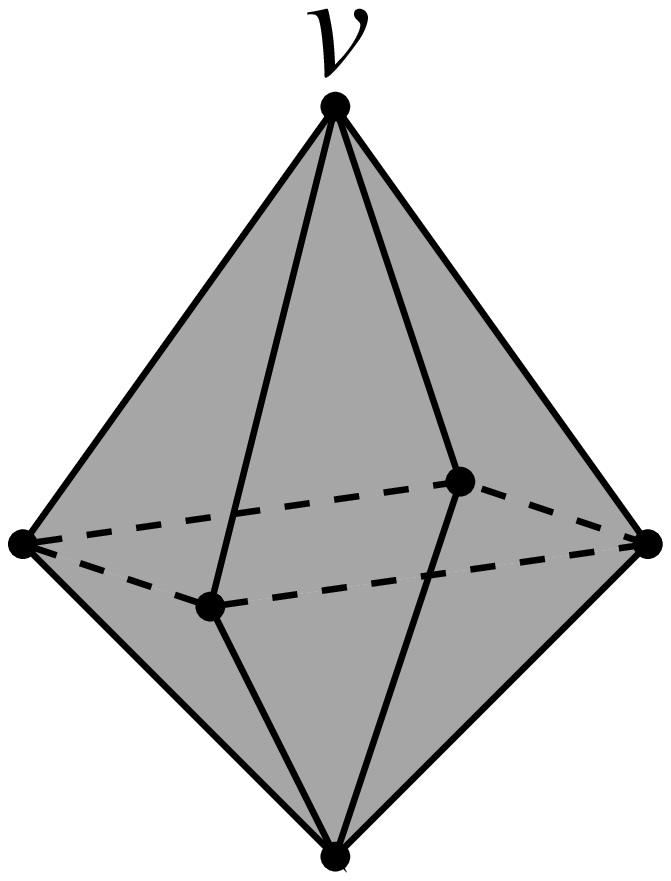}
    \caption {$2$-dimensional cycle containing a $1$-dimensional cycle.} \label{fig:cycle_within_cycle}

}
\end{figure}

\begin{prop}[{\bf $d$-Dimensional cycles contain $(d-1)$-dimensional cycles}] \label{prop:dcycle_contains_smallercycle}
Let $v$ be a vertex of a $d$-dimensional cycle $\Omega$ and let $F_1,\ldots,F_k$ are the $d$-faces of $\Omega$ containing $v$. The $(d-1)$-path-connected components of the simplicial complex
\[
    \langle F_1\setminus \{v\}, \ldots, F_k\setminus \{v\}\rangle
\]
are $(d-1)$-dimensional cycles.
\end{prop}

\begin{proof}
Let $\Omega_v = \langle F_1\setminus \{v\}, \ldots, F_k\setminus \{v\}\rangle$ and let $\Omega_v'$ be a $(d-1)$-path-connected component of $\Omega_v$ with $(d-1)$-faces $F_{i_1}\setminus \{v\}, \ldots, F_{i_\ell}\setminus \{v\}$.  To show that $\Omega_v'$ is a $(d-1)$-dimensional cycle we need only show that each of its $(d-2)$-faces is contained in an even number of the faces $F_{i_1}\setminus \{v\}, \ldots, F_{i_\ell}\setminus \{v\}$.  Let $f$ be a $(d-2)$-face of $\Omega_v'$.  Note that, for $1 \leq r \leq \ell$, $f$ is a $(d-2)$-face of $F_{i_r}\setminus \{v\}$ if and only if $f \cup \{v\}$ is a $(d-1)$-face of $F_{i_r}$.  The face $f \cup \{v\}$ belongs to an even number of the $d$-faces in $\Omega$ since $\Omega$ is a $d$-dimensional cycle.  Since these faces all contain $v$ they belong to the set $\{F_1,\ldots,F_k\}$ and thus $f \cup \{v\}$ belongs to an even number of the $d$-faces $F_1,\ldots,F_k$.  Note that, after removing $v$ from these $d$-faces, they are $(d-1)$-path-connected in $\Omega_v$ since they all contain $f$.  Hence, with $v$ removed, these faces all lie in $\Omega_v'$ and so $f$ is contained in an even number of the faces $F_{i_1}\setminus \{v\}, \ldots, F_{i_\ell}\setminus \{v\}$.  Therefore $\Omega_v'$ is a $(d-1)$-dimensional cycle.
\end{proof}

In contrast to Proposition \ref{prop:dcycle_contains_smallercycle} we can also create higher-dimensional cycles from lower-dimensional ones under certain homological conditions over $\Z_2$ by joining an additional vertex to the facets of the cycle.  For an illustration of this construction see Figure \ref{fig:cone_construction}.

\begin{figure}[h]
\centering
\subfloat[The $1$-dimensional cycle $\Omega$]{\makebox[5cm]{
            \includegraphics[height=1.2in]{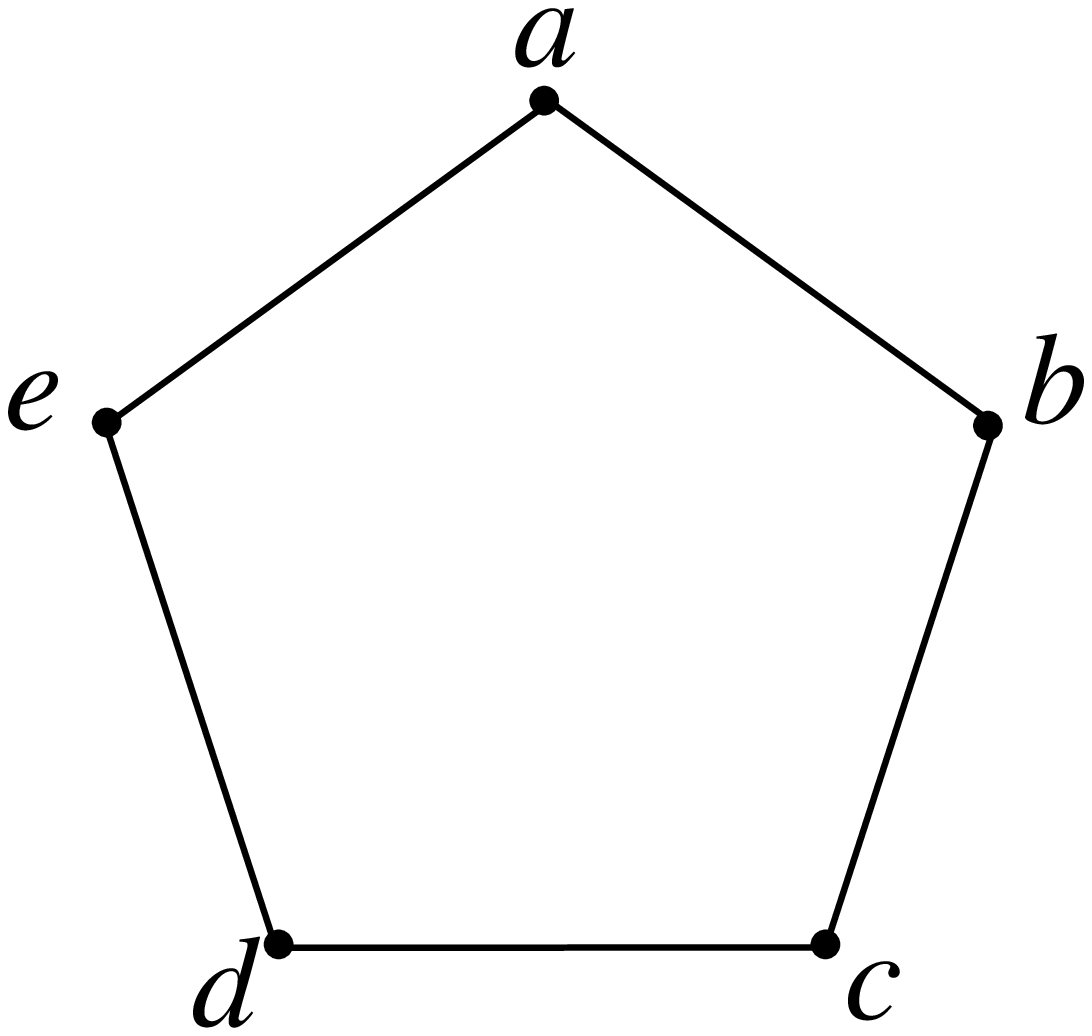}}
            \label{fig:cone_construction1}} \qquad
\subfloat[$\Omega$ as the support complex of a $1$-boundary and the vertex $v$]{\makebox[5cm]{
	\includegraphics[height=1.2in]{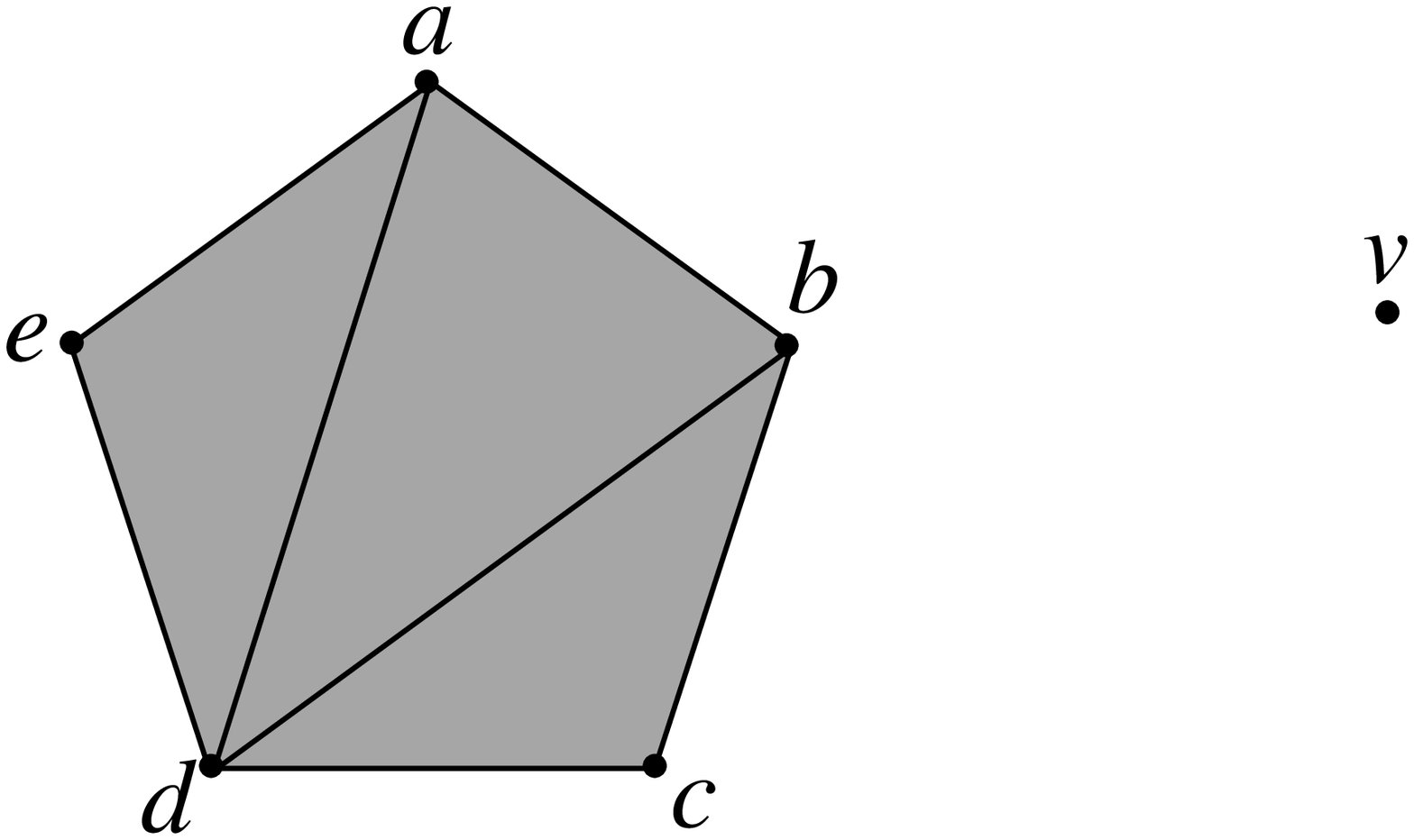}}
	\label{fig:cone_construction2}} \\
\subfloat[The $2$-dimensional cycle $\Phi$]{\makebox[5cm]{
	\includegraphics[height=1.2in]{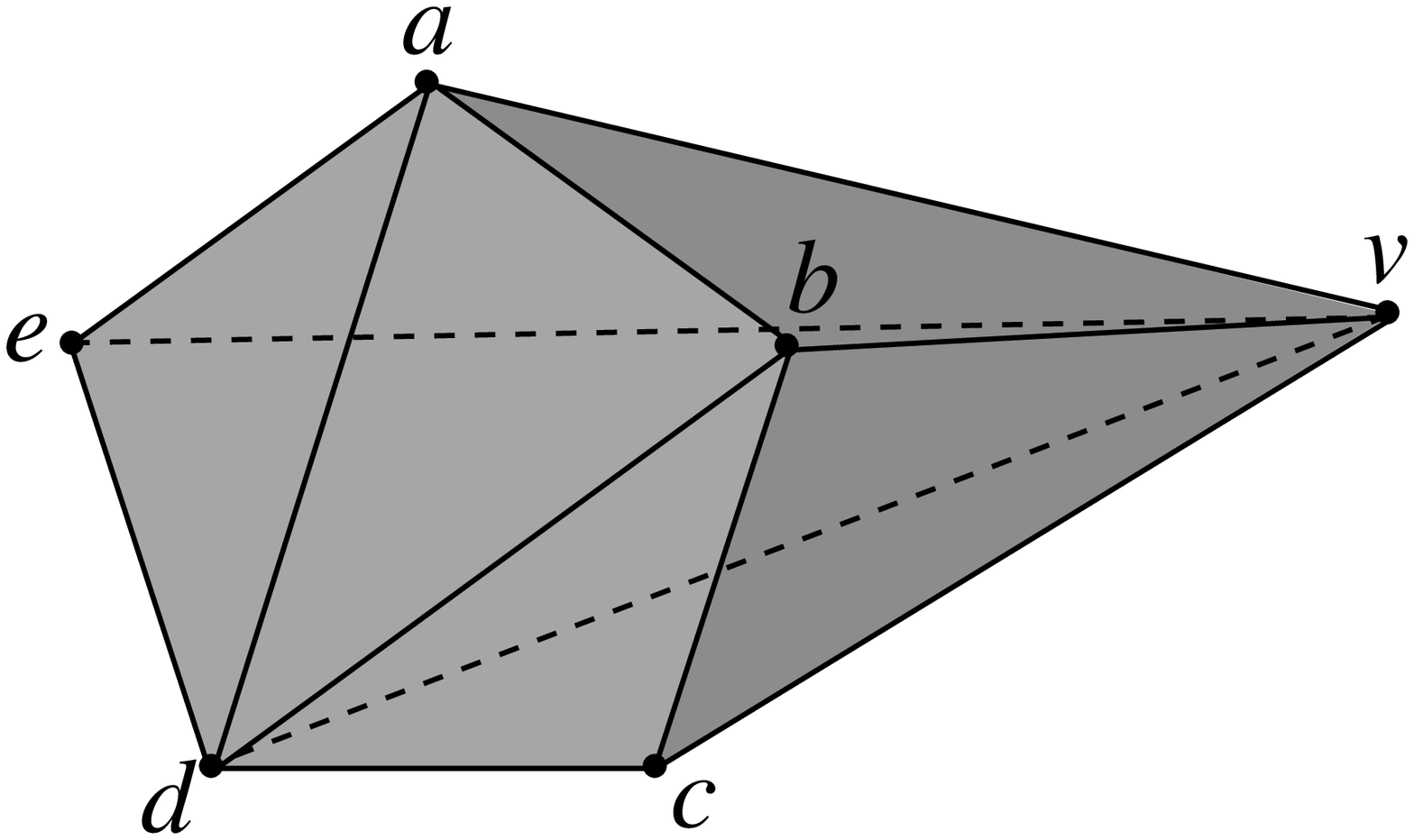}}
	\label{fig:cone_construction3}}
\caption{Example of construction in Proposition \ref{prop:cone_cycle}.} \label{fig:cone_construction}
\end{figure}

\begin{prop} [{\bf $d$-Dimensional cycles extend to $(d+1)$-dimensional cycles in some cases}] \label{prop:cone_cycle}
Let $\Delta$ be a simplicial complex containing the $d$-dimensional cycle $\Omega$ whose $d$-faces are $F_1,\ldots,F_k$. If there exist $(d+1)$-faces $A_1,\ldots,A_\ell$ in $\Delta_{V(\Omega)}$ such that, over $\Z_2$ we have
\begin{equation}\label{eq:boundaryofcycle}
    \partial_{d+1}\left(\sum_{i=1}^\ell A_i\right)=\sum_{j=1}^k F_j
\end{equation}
and no strict subset of $\{A_1,\ldots,A_\ell\}$ also satisfies (\ref{eq:boundaryofcycle}) then
\[
    \Phi = \langle F_1 \cup v,\ldots,F_k\cup v, A_1,\ldots,A_\ell \rangle
\]
is a $(d+1)$-dimensional cycle for any vertex $v \notin V(\Omega)$.
\end{prop}

\begin{proof}
We will first show that each $d$-face of $\Phi$ is contained in an even number of the $(d+1)$-faces of $\Phi$.  Let $f$ be any $d$-face of $\Phi$.  We have three cases to consider.

First suppose that $v \in f$.  In this case, $f$ is not contained in any of the $A_i$'s and so $f$ is a subset of $F_{j}\cup \{v\}$ for some $j$.  So we have $f\setminus \{v\} \subseteq F_{j}$.  Since $\Omega$ is a $d$-dimensional cycle $f\setminus \{v\}$ belongs to an even number of the $d$-faces $F_1,\ldots,F_k$.  Therefore $f$ belongs to an even number of the $(d+1)$-faces $F_{i}\cup \{v\}$ and so belongs to an even number of the $(d+1)$-faces of $\Phi$.

Now suppose that $v \notin f$ and that $f$ belongs to at least one $(d+1)$-face of the form $F_{j}\cup \{v\}$ for some $j$. In this case we must have $f = F_j$ and so $f \not \subseteq F_i \cup \{v\}$ for any $i\neq j$.  Thus $f$ appears exactly once on the right-hand-side of (\ref{eq:boundaryofcycle}) and since this equation holds over $\Z_2$, $f$ must be contained in an odd number of the $A_i$'s.  Hence overall $f$ is contained in an even number of the $(d+1)$-faces of $\Phi$.

Finally suppose that $v \notin f$ and that $f$ does not belong to any $(d+1)$-faces of the form $F_{j}\cup \{v\}$.  Then $f$ is not a $d$-face of $\Omega$.  Again, since (\ref{eq:boundaryofcycle}) holds over $\Z_2$, we know that $f$ belongs to an even number of the $A_i$'s.  Thus $f$ is contained in an even number of the $(d+1)$-faces of $\Phi$.

Therefore we know that the $(d+1)$-path-connected components of $\Phi$ are $(d+1)$-dimensional cycles.  Note that the $(d+1)$-faces $F_i\cup \{v\}$ all lie in the same $(d+1)$-path-connected component of $\Phi$ since $\Omega$ is $d$-path-connected.  Recall from above that, for any $j$, a $d$-face belonging to $F_j\cup \{v\}$ which does not contain $v$ is equal to $F_j$ and must belong to at least one of the $A_i$'s by (\ref{eq:boundaryofcycle}).  Thus at least one of the $A_i$'s belongs to the $(d+1)$-path-connected component of $\Phi$ which contains the $F_j\cup \{v\}$'s.  Therefore if $\Phi$ has any other $(d+1)$-path-connected component then it consists solely of a strict subset of the $A_i$'s.  Without loss of generality let these faces be $A_1,\ldots,A_r$ where $r < \ell$.  We know that these $(d+1)$-path-connected components are all $(d+1)$-dimensional cycles and so, as we will see in Proposition \ref{prop:ddimcycle_is_dcycle}, which says that the sum of the $(d+1)$-faces of such a cycle form a homological $(d+1)$-cycle, we have that
\[
    \partial_{d+1} \left( \sum_{j=1}^{r} A_j \right) =0
\]
and so
\[
    \partial_{d+1} \left( \sum_{j=1}^{\ell} A_j \right) = \partial_{d+1} \left( \sum_{j=1}^{r} A_j \right) + \partial_{d+1} \left( \sum_{j=r+1}^{\ell} A_j \right) = \partial_{d+1} \left( \sum_{j=r+1}^{\ell} A_j \right).
\]
Therefore by (\ref{eq:boundaryofcycle}) we have
\[
   \partial_{d+1} \left( \sum_{j=r+1}^{\ell} A_{j} \right)= \sum_{i=1}^k F_i
\]
which contradicts the minimality of our choice of $A_1,...,A_\ell$.  Hence $\Phi$ has only one $(d+1)$-path-connected component and so it is a $(d+1)$-dimensional cycle.
\end{proof}

The notion of a $d$-dimensional cycle extends the classical concept of a pseudo $d$-manifold that appears in algebraic topology.  See for example \cite{Munk84}.

\begin{definition}[{\bf pseudo $d$-manifold}]
A pure $d$-dimensional $d$-path-connected simplicial complex $\Delta$ is a {\bf pseudo $d$-manifold}  if every $(d-1)$-face of $\Delta$ is contained in exactly two $d$-faces of $\Delta$.
\end{definition}

Notice that the $2$-dimensional cycles in Figures \ref{fig:octacycle} through \ref{fig:real_proj_plane} are all examples of pseudo $2$-manifolds. The simplicial complexes in Figures \ref{fig:two_pyramid} and \ref{fig:pinched_sphere} are not pseudo $2$-manifolds as they have $1$-faces belonging to more than two $2$-faces.

Notice that a graph cycle is a $1$-dimensional cycle, but a $1$-dimensional cycle need not be a graph cycle.  See Figure \ref{fig:1_dim_cycle} for an example.  With this in mind, we introduce a notion of minimality into the idea of a $d$-dimensional cycle.

\begin{figure}[h!]
{\centering
    \includegraphics[height=1.1in]{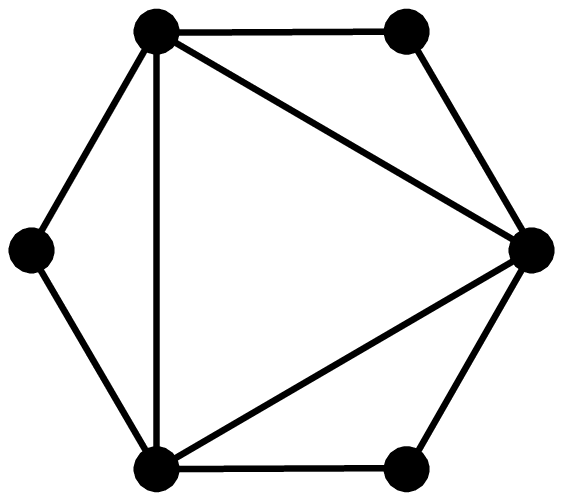}
    \caption {A $1$-dimensional cycle which is not a graph cycle.} \label{fig:1_dim_cycle}

}
\end{figure}

\begin{definition}[{\bf face-minimal $d$-dimensional cycle}]
A $d$-dimensional cycle $\Omega$ is {\bf face-minimal} when there are no $d$-dimensional cycles on a strict subset of the $d$-faces of $\Omega$.
\end{definition}

Notice that a $1$-dimensional cycle is a graph cycle if and only if it is face-minimal.  In Figure \ref{fig:2d_cycles} the only $2$-dimensional cycle which is not face-minimal is \ref{fig:two_pyramid}.

\begin{lem}[{\bf $d$-Dimensional cycles can be partitioned into face-minimal cycles}] \label{lem:cycle_decomp}
Any $d$-dimensional cycle $\Omega$ can be written as a union of face-minimal $d$-dimensional cycles $\Phi_1,\ldots,\Phi_n$ such that every $d$-face of $\Omega$ belongs to some $\Phi_i$ and such that $\Phi_i$ and $\Phi_j$ have no $d$-faces in common when $i \neq j$.
\end{lem}

\begin{proof}
If $\Omega$ is a face-minimal $d$-dimensional cycle then we are done.  So suppose that $\Omega$ is not face-minimal and let $\Phi_1$ be a face-minimal $d$-dimensional cycle on a strict subset of the $d$-faces of $\Omega$.  Consider the $d$-path-connected components of the complex $\Omega_1$ whose facets are the $d$-faces of $\Omega$ not belonging to $\Phi_1$. We claim that each such component is a $d$-dimensional cycle.  Since each component is $d$-path-connected by definition, we need only show that each $(d-1)$-face of the complement is contained in an even number of $d$-faces.

Let $\Psi$ be one of the $d$-path-connected components of $\Omega_1$ and let $f$ be a $(d-1)$-face of $\Psi$.  Suppose first that $f$ also belongs to one of the $d$-faces of $\Phi_1$.  Since $\Phi_1$ is a $d$-dimensional cycle $f$ belongs to an even number of its $d$-faces.  However $\Omega$ is also a $d$-dimensional cycle and $f$ belongs to an even number of its $d$-faces.  Since $\Omega_1$ is the complex whose facets are the $d$-faces of $\Omega$ not in $\Phi_1$ $f$ belongs to an even number of $d$-faces in $\Omega_1$.  However the collection of all $d$-faces of $\Omega_1$ containing $f$ is clearly $d$-path-connected and so these $d$-faces all lie in $\Psi$.  Hence $f$ belongs to an even number of $d$-faces of $\Psi$.

If $f$ does not belong to any $d$-faces of $\Phi_1$ then all of the $d$-faces of $\Omega$ which contain $f$ lie in $\Omega_1$.  Since there are an even number of such faces and they are $d$-path-connected they all lie in $\Psi$.  Hence $f$ belongs to an even number of $d$-faces of $\Psi$.  Therefore each $d$-path-connected component of $\Omega_1$ is a $d$-dimensional cycle.

Each of these components is either a face-minimal $d$-dimensional cycle, or contains a face-minimal $d$-dimensional cycle on a strict subset of its $d$-faces.  We may repeat the argument above on the simplicial complex whose facets are the $d$-faces of $\Omega_1$ belonging to the components that are not face-minimal cycles.  Iterating this procedure we see that, since we have a finite number of $d$-faces, eventually the procedure must terminate.  We are left with face-minimal cycles $\Phi_1,\ldots,\Phi_n$ in which every $d$-face of $\Omega$ belongs to some $\Phi_i$ and, by our construction, no two distinct cycles $\Phi_i$ and $\Phi_j$ share a $d$-face.
\end{proof}

\begin{prop}
A pseudo $d$-manifold is a face-minimal $d$-dimensional cycle.
\end{prop}

\begin{proof}
Let $\Delta$ be a pseudo $d$-manifold.  Then $\Delta$ is $d$-path connected and every $(d-1)$-face in $\Delta$ belongs to exactly two $d$-faces.  Hence $\Delta$ is a $d$-dimensional cycle.  Suppose that $\Delta$ is not face-minimal.  By Lemma \ref{lem:cycle_decomp} we can partition $\Delta$ into face-minimal $d$-dimensional cycles $\Phi_1,\ldots,\Phi_n$ where $n \geq 2$.  Since $\Delta$ is $d$-path connected there must exist some pair of indices $i,j$ with $i \neq j$ with $F_1 \in \Phi_i$ and $F_2 \in \Phi_j$ where $F_1 \cap F_2 = f$ for some $(d-1)$-face $f$ of $\Delta$.  Since $\Phi_i$ is a $d$-dimensional cycle then $f$ belongs to an even number of $d$-faces of $\Phi_i$ and similarly $f$ belongs to an even number of $d$-faces of $\Phi_j$.  By Lemma \ref{lem:cycle_decomp}, these $d$-faces are all distinct which means that $f$ belongs to at least four $d$-faces of $\Delta$.  This is a contradiction since $\Delta$ is a pseudo $d$-manifold.  Hence $\Delta$ is a face-minimal $d$-dimensional cycle.
\end{proof}

The converse of this theorem does not hold.  The simplicial complex in Figure \ref{fig:pinched_sphere}, a triangulated sphere pinched along a $1$-dimensional face, is a counter-example.  It is a face-minimal $2$-dimensional cycle, but it is not a pseudo $2$-manifold as it has a $1$-dimensional face, $\{x,y\}$, belonging to four distinct $2$-dimensional faces.

A pseudo $d$-manifold can be classified as either {\bf orientable} or {\bf non-orientable} and this idea can be generalized to the case of $d$-dimensional cycles.  Recall from Section \ref{sec:prelim} that an orientation of a face in a simplicial complex is simply an ordering of its vertices.

\begin{definition}[{\bf induced orientation}]
Let $F$ be an oriented $d$-face of a simplicial complex $\Omega$ and let $v$ be any vertex of $\Omega$ contained in $F$. The {\bf induced orientation} of the $(d-1)$-face $F\setminus\{v\}$ is given in the following way:
\begin{itemize}
\item if $v$ is in an odd position of the ordering of the vertices of $F$ then the orientation of $F\setminus\{v\}$ is given by the ordering of its vertices in the orientation of $F$
\item if $v$ is in an even position of the ordering of the vertices of $F$ then the orientation of $F\setminus\{v\}$ is given by any odd permutation of the ordering of its vertices in the orientation of $F$
\end{itemize}
\end{definition}

\begin{example}
Let $\Delta$ be a simplicial complex containing the oriented $4$-face $[a,b,c,d,e]$.  The induced orientation of the $3$-face $\{a,b,d,e\}$ is $a<b<d<e$ and the induced orientation of the $3$-face $\{a,c,d,e\}$ is $a<c<e<d$.
\end{example}

Notice that a $(d-1)$-face which belongs to more than one oriented $d$-face in a simplicial complex will have an induced orientation corresponding to each oriented $d$-face to which it belongs.  These induced orientations may be non-equivalent.

\begin{definition}[{\bf orientable $d$-dimensional cycle}]\label{def:orient_ddim_cycle}
Let $\Omega$ be a $d$-dimensional cycle.  If it is possible to choose orientations of the $d$-faces of $\Omega$ such that for any $(d-1)$-face of $\Omega$ its induced orientations are divided equally between the two orientation classes then we say that $\Omega$ is {\bf orientable}.  Otherwise $\Omega$ is {\bf non-orientable}.
\end{definition}

Note that when we talk about the oriented $d$-faces of an orientable $d$-dimensional cycle we are referring to any set of orientations that is consistent with Definition \ref{def:orient_ddim_cycle}.

Many of the combinatorial complexities that exist in higher dimensions are not present in the $1$-dimensional case.  Non-orientable cycles are an example of this.

\begin{prop} \label{prop:1_dim_orient}
Any $1$-dimensional cycle is orientable.
\end{prop}

\begin{proof}
Let $\Omega$ be a $1$-dimensional cycle.  If $\Omega$ is face-minimal then it is a graph cycle.  It is straightforward to see that any graph cycle is orientable by choosing a ``direction'' in which to traverse the cycle and orienting each face in a way that is consistent with this direction.

If $\Omega$ is not face-minimal then by Lemma \ref{lem:cycle_decomp} we can partition the $1$-faces of $\Omega$ into face-minimal $1$-dimensional cycles $\Phi_1,\ldots,\Phi_n$ where $n \geq 2$.  For each $1\leq i \leq n$, $\Phi_i$ is orientable.  Let $v$ be any vertex of $\Omega$.  Then $v$ belongs to some subset of the cycles $\Phi_1,\ldots,\Phi_n$.  In each such cycle there are two induced orientations of $v$ and they are opposite to each other.  Thus overall the induced orientations of $v$ in $\Omega$ are divided equally between the two orientation classes.  Therefore $\Omega$ is orientable.
\end{proof}

The $2$-dimensional cycles given in Figure \ref{fig:2d_cycles} are all examples of orientable $2$-dimensional cycles except for the triangulation of the real projective plane given in Figure \ref{fig:real_proj_plane} which is a non-orientable $2$-dimensional cycle.

\Comment{One particularly simple, orientable $d$-dimensional cycle is the boundary of a $(d+1)$-simplex.  We denote the $d$-dimensional $d$-complete complex on $n$ vertices by $\Lambda_n^d$.}

One particularly simple, orientable $d$-dimensional cycle is the boundary of a $(d+1)$-simplex.  We denote the $d$-dimensional $d$-complete complex on $n$ vertices by $\Lambda_n^d$.

\begin{prop} [{\bf The smallest $d$-dimensional cycle is the boundary of a $(d+1)$-simplex}] \label{prop:smallest_cycle}
A $d$-dimensional cycle can have no fewer than $d+2$ vertices.  If $\Omega$ is a $d$-dimensional cycle on $d+2$ vertices then $\Omega = \Lambda_{d+2}^{d}$.  In addition, $\Lambda_{d+2}^d$ is an orientable $d$-dimensional cycle.
\end{prop}

\begin{proof}
It is clear that any $d$-dimensional cycle must have at least $d+2$ vertices since a single $d$-face contains $d+1$ vertices.  It is easy to see that any two $d$-faces of $\Lambda_{d+2}^{d}$ have $d$ vertices in common and so are connected by a $d$-path.  Also, each set of $d$ vertices in $\Lambda_{d+2}^{d}$ belongs to exactly two $d$-faces.  Hence $\Lambda_{d+2}^{d}$ is a $d$-dimensional cycle.

Conversely, let $\Omega$ be any $d$-dimensional cycle on $d+2$ vertices.  There are only $d+2$ possible $d$-faces on a set of $d+2$ vertices and so in order to show that $\Omega$ is $d$-complete we must show that it has $d+2$  distinct $d$-faces.  Let $F$ be any $d$-face of $\Omega$ and let $f$ be one of its $(d-1)$-faces.  We know that $f$ must belong to at least one other $d$-face of $\Omega$ since it is a $d$-dimensional cycle.  There is only one vertex $v$ of $\Omega$ not already contained in $F$ and so $f \cup \{v\}$ must be a $d$-face of $\Omega$.   Since $F$ contains $d+1$ of these distinct $(d-1)$-faces which all must lie in another $d$-face of $\Omega$ this gives rise to $d+1$ distinct $d$-faces of $\Omega$ which all contain $v$.  Therefore $\Omega$ contains $d+2$ distinct $d$-faces including $F$. Hence $\Omega$ is $d$-complete and so $\Omega = \Lambda_{d+2}^d$.

We would like to show that $\Lambda_{d+2}^d$ is orientable.  Each $(d-1)$-face of $\Lambda_{d+2}^d$ belongs to just two $d$-faces and so we need to ensure that there is a way to orient the $d$-faces of $\Lambda_{d+2}^d$ so that the orientations induced on each $(d-1)$-face are opposite to each other.

We propose assigning orientations to the $d$-faces in the following way.  Suppose that $v_1,\ldots,v_{d+2}$ are the vertices of $\Lambda_{d+2}^d$. Let the orientation of each $d$-face be the induced orientation that results from thinking of the $d$-face as a subface of the oriented simplex $[v_1,\ldots,v_{d+2}]$.

Let $f$ be any $(d-1)$-face of $\Lambda_{d+2}^d$.  We know that $f$ belongs to exactly two $d$-faces of $\Lambda_{d+2}^d$ which we will call $F$ and $G$. Since $\Lambda_{d+2}^d$ has $d+2$ vertices then we have $V(\Lambda_{d+2}^d) \setminus f = \{v_s,v_t\}$ for some $1\leq s < t \leq d+2$ where, without loss of generality, we have $v_s \in F$ and $v_t \in G$.  Since $s < t$ we know that $v_s$ appears before $v_t$ in the ordering above.  The ordering induced on $f$ by $F$ is achieved by first removing $v_t$ from $[v_1,\ldots,v_{d+2}]$ to induce an ordering on $F$ and then removing $v_s$ to induce an ordering on $f$, applying odd permutations where necessary.  The ordering induced on $f$ by $G$ is achieved by first removing $v_s$ from $[v_1,\ldots,v_{d+2}]$ to induce an ordering on $G$ and then removing $v_t$ to induce an ordering on $f$, again applying odd permutations where necessary.  Note that the removal of $v_t$ does not change whether or not $v_s$ is in an even or odd position in the ordering since it appears before $v_t$.  However removing $v_s$ before $v_t$ causes $v_t$ to move either from an even position to an odd one or from an odd one to an even one.  Consequently the orientation of $f$ induced by $F$ is necessarily an odd permutation of the orientation induced by $G$.  Therefore the two orientations of $f$ are opposite.  Hence $\Lambda_{d+2}^d$ is an orientable $d$-dimensional cycle.
\end{proof}

\begin{example}
The hollow tetrahedron $\Lambda_4^2$ is shown in Figure \ref{fig:hollow_tetra}.  It is the boundary of a $3$-simplex and, by Proposition \ref{prop:smallest_cycle}, it is the $2$-dimensional cycle on the smallest number of vertices.
\end{example}

\begin{figure}[b]
{\centering
    \includegraphics[height=1.2in]{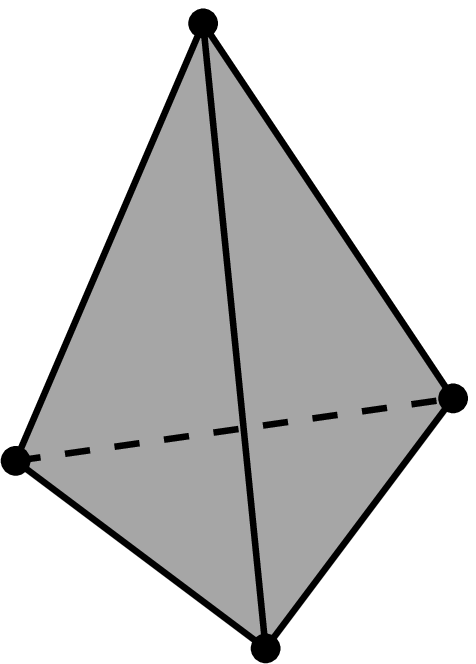}
    \caption {The $2$-dimensional $2$-complete simplicial complex.} \label{fig:hollow_tetra}
}
\end{figure}


\section{A Combinatorial Condition for Non-zero Homology}
In this section we will show that a $d$-dimensional cycle is the right combinatorial structure to describe the idea of non-zero $d$-dimensional homology over a field of characteristic 2.  We will also show that orientable $d$-dimensional cycles produce non-zero $d$-dimensional homology over any field.

We begin by investigating the relationship between the combinatorial structure of a simplicial complex and its simplicial homology over $\Z_2$.  In this field the role played by the coefficients in a $d$-chain is reduced to indicating whether or not a face is present.  As well since $-1=1$ over $\Z_2$ the concept of an orientation of a face is unnecessary as all orientations of a face are equivalent.  This allows us to more easily examine the connections between the combinatorics of the simplicial complex and the more algebraic concepts of the $d$-cycle and the $d$-boundary.  In fact, over $\Z_2$, we lose no information when we translate between the $d$-chain $\sum_{i=1}^m F_i$ and the support complex $\langle F_1,\ldots,F_m \rangle$.  This makes the field $\Z_2$ an ideal setting to investigate the correspondence between complexes which generate non-zero simplicial homology and their combinatorial properties.

\Comment{We begin by investigating the relationship between the combinatorial structure of a simplicial complex and its simplicial homology over $\Z_2$.  In this field the role played by the coefficients in a $d$-chain is reduced to indicating whether or not a face is present.  As well since $-1=1$ over $\Z_2$ the concept of an orientation of a face is unnecessary as all orientations of a face are equivalent.  This allows us to more easily examine the connections between the combinatorics of the simplicial complex and the more algebraic concepts of the $d$-cycle and the $d$-boundary.}

The following proposition demonstrates the relationship between $d$-dimensional cycles and homological $d$-cycles over $\Z_2$.

\begin{prop} [{\bf Relationship between $d$-dimensional cycles and homological $d$-cycles}] \label{prop:ddimcycle_is_dcycle}
If $\Omega$ is a $d$-dimensional cycle with $d$-faces $F_1,\ldots,F_k$ then $\sum_{i=1}^k F_i$ is a homological $d$-cycle over $\Z_2$.  Conversely, if $\sum_{i=1}^k F_i$ is a homological $d$-cycle over $\Z_2$ then the $d$-path-connected components of $\langle F_1,\ldots, F_k\rangle$ are $d$-dimensional cycles.
\end{prop}

\begin{proof}
Let $\Omega$ be a $d$-dimensional cycle with $d$-faces $F_1,\ldots,F_m$.  Setting $c=\sum_{i=1}^m F_i$ and applying the boundary map $\partial_d$ over $\Z_2$ we have
\[
    \partial_d (c) = \sum_{i=1}^m (e^{i}_1+\cdots +  e^{i}_{d+1})
\]
where $e^{i}_1,\ldots,e^{i}_{d+1}$ are the $d+1$ edges of dimension $d-1$ belonging to $F_i$.  Since the faces $F_1,\ldots,F_m$ form a $d$-dimensional cycle each $(d-1)$-dimensional face appears in an even number of the faces $F_1,\ldots,F_m$.  Hence, since our coefficients belong to $\Z_2$, we have $\partial_d (c) = 0$ and so $c$ is a $d$-cycle.

Let $c=F_1+\cdots +F_m$ be a $d$-cycle over $\Z_2$.  Applying the boundary map $\partial_d$ we have
\[
    0=\partial_d (F_1+\cdots +F_m) = \sum_{i=1}^m (e^{i}_1+\cdots +  e^{i}_{d+1})
\]
where $e^{i}_1,\ldots,e^{i}_{d+1}$ are the $d+1$ edges of dimension $d-1$ belonging to $F_i$.  If the support complex of $c$ is not $d$-path-connected then we can partition this complex into $n$ $d$-path-connected components $\Phi_1,\ldots,\Phi_n$ where $n \geq 2$.
Let $P_i \subseteq \{1,\ldots,m\}$ be such that $F_j \in \Phi_i$ if and only if $j \in P_i$.  Note that $P_1,\ldots,P_n$ form a partition of $\{1,\ldots,m\}$.  Since $\Phi_1,\ldots,\Phi_n$ have no $(d-1)$-faces in common, for each $1\leq i \leq n$ we must have
\[
    \sum_{j \in P_i} (e^{j}_1+\cdots+e^{j}_{d+1})=0
\]
and so, since our sum is over $\Z_2$, we see that each $(d-1)$-face occurring in $\Phi_i$ belongs to an even number of the $d$-faces in $\{F_j | j \in P_i\}$.  Therefore $\Phi_i$ is a $d$-dimensional cycle and hence the $d$-path-connected components of the support complex of $c$ all form $d$-dimensional cycles.
\end{proof}

Due to the close association between homological $d$-cycles and $d$-dimensional cycles over the field $\Z_2$ we are able to obtain a necessary and sufficient combinatorial condition for non-zero homology over any field of characteristic 2.

\begin{thm} [{\bf When simplicial homology vanishes in characteristic 2}]  \label{thm:cycle_hom_z2}
Let $\Delta$ be a simplicial complex and let $k$ be a field of characteristic $2$. Then $\tilde{H}_d(\Delta;k) \neq 0$ if and only if $\Delta$ contains a $d$-dimensional cycle, the sum of whose $d$-faces is not a $d$-boundary.
\end{thm}

\begin{proof}
By application of the Universal Coefficient Theorem given in \cite[Theorem 3A.3]{Hatch02} and the property of faithful flatness we find that for any field $k$ of characteristic $2$ we have $\tilde{H}_d(\Delta;k) \neq 0$ if and only if $\tilde{H}_d(\Delta;\Z_2) \neq 0$.  Therefore we need only prove the theorem in the case that $k = \Z_2$.

Suppose that $\tilde{H}_d(\Delta;\Z_2) \neq 0$.  Then $\Delta$ contains a $d$-cycle $c$ that is not a $d$-boundary.  We may assume that the support complex of $c$ is minimal with respect to this property.  In other words no strict subset of the $d$-faces of $c$ has a sum that is also a $d$-cycle which is not a $d$-boundary.  First we would like to show that the support complex of $c$ is $d$-path connected.

Since we are in $\Z_2$, the $d$-cycle $c$ is of the form
\[
    c=F_1+\cdots +F_m
\]
for some $d$-faces $F_1,\ldots,F_m$ of $\Delta$.  Since $c$ is a $d$-cycle then applying the boundary map $\partial_d$ we have
\[
    0=\partial_d (F_1+\cdots +F_m) = \sum_{i=1}^m (e^{i}_1+\cdots +  e^{i}_{d+1})
\]
where $e^{i}_1,\ldots,e^{i}_{d+1}$ are the $d+1$ edges of dimension $d-1$ belonging to $F_i$.  If the support complex of $c$ is not $d$-path-connected then, without loss of generality, we can partition its set of $d$-faces into two sets, $\{F_1,\ldots,F_\ell\}$ and $\{F_{\ell+1}, \ldots,F_m\}$ such that these two sets have no $(d-1)$-faces in common.  Hence we must have
\[
    \sum_{i=1}^{\ell} (e^{i}_1+\cdots+e^{i}_{d+1})=0 \qquad \textrm{and} \qquad \sum_{i=\ell+1}^{m} (e^{i}_1+\cdots+e^{i}_{d+1})=0.
\]
In other words we have
\[
    \partial_d (F_1+\cdots+F_\ell) =  0 \qquad \textrm{and} \qquad \partial_d (F_{\ell+1}+\cdots+F_m)  = 0
\]
and so $F_1+\cdots+F_\ell$ and $F_{\ell+1}+\cdots+F_m$ are both $d$-cycles.  By our assumption of minimality these $d$-cycles are both $d$-boundaries.  Hence in $\Delta$ there exist $(d+1)$-faces $G_1,\ldots,G_r$ and $H_1, \ldots,H_t$ such that
\[
    \partial_{d+1}\left(\sum_{i=1}^r G_i\right)=F_1+\cdots+F_\ell \qquad \textrm{and} \qquad \partial_{d+1}\left(\sum_{i=1}^t H_i\right)=F_{\ell + 1}+\cdots+F_m.
\]
But then we have
\[
    \partial_{d+1}\left(\sum_{i=1}^r G_i+\sum_{i=1}^t H_i\right)=F_1+\cdots +F_m
\]
which is a contradiction since $F_1+\cdots +F_m$ is not a $d$-boundary.  Therefore the support complex of $c$ must be $d$-path-connected.  Hence, by Proposition \ref{prop:ddimcycle_is_dcycle} the support complex of $c$ is a $d$-dimensional cycle.  Hence $\Delta$ contains a $d$-dimensional cycle the sum of whose $d$-faces is not a $d$-boundary.

Conversely suppose that $\Delta$ contains a $d$-dimensional cycle with $d$-faces $F_1,\ldots,F_m$ such that $\sum_{i=1}^m F_i$ is not a $d$-boundary.  By Proposition \ref{prop:ddimcycle_is_dcycle} we know that $\sum_{i=1}^m F_i$ is a $d$-cycle.  It follows that $\tilde{H}_d(\Delta;\Z_2) \neq 0$.
\end{proof}

When we broaden the scope of our investigations to study simplicial homology over an arbitrary field we must keep in mind examples such as the triangulation of the real projective plane given in Figure \ref{fig:real_proj_plane}.  The simplicial homology of this complex changes significantly depending on the field under consideration.  In particular this complex has non-zero $2$-dimensional homology only over fields of characteristic 2.  As we saw in Section \ref{sec:ddimcycles} the triangulation of the real projective plane is an example of a non-orientable $2$-dimensional cycle.  It is this notion of orientability which leads us to a sufficient condition for a simplicial complex to have non-zero homology over any field.  First we see that orientable $d$-dimensional cycles are homological $d$-cycles over any field.

\begin{lem} [{\bf Orientable $d$-dimensional cycles are $d$-cycles}] \label{lem:orient_ddimcycle_is_dcycle}
The sum of the oriented $d$-faces of an orientable $d$-dimensional cycle is a homological $d$-cycle over any field $k$.
\end{lem}

\begin{proof}
Let $\Omega$ be an orientable $d$-dimensional cycle with $d$-faces $F_1,\ldots,F_m$ and let
\[
    c= F_1+\cdots+F_m
\]
where $F_1,\ldots,F_m$ are given orientations consistent with Definition \ref{def:orient_ddim_cycle}.  For $1 \leq i \leq m$ let $e^{i}_1,\ldots,e^{i}_{d+1}$ be the $d+1$ faces of dimension $d-1$ belonging to $F_i$.  Applying the boundary map to $c$ we have, without loss of generality,
\begin{equation}\label{eq:sum_dfaces}
    \partial_d(c) = \sum_{i=1}^m\sum_{j=1}^{d+1}(-1)^{j+1}e^i_j.
\end{equation}
Notice that every $(d-1)$-face of $\Omega$ occurs an even number of times in (\ref{eq:sum_dfaces}) since $\Omega$ is a $d$-dimensional cycle.  Furthermore, because $\Omega$ is orientable the number of times that the $(d-1)$-face appears with a positive sign is equal to the number of times that it appears with a negative sign.  Hence we get $\partial_d(c) = 0$ and so $c$ is a homological $d$-cycle.
\end{proof}

\begin{thm}  [{\bf Orientable $d$-dimensional cycles give non-zero homology over all fields}] \label{thm:orientable_cycle_hom}
If a simplicial complex $\Delta$ contains an orientable $d$-dimensional cycle, the sum of whose oriented $d$-faces is not a $d$-boundary, then $\tilde{H}_d(\Delta;k) \neq 0$ for any field $k$.
\end{thm}

\begin{proof}
\Comment{If $d=1$ then the statement follows from Theorem \ref{thm:1hom_graph_cycle} by observing that any orientable $1$-dimensional cycle contains a graph cycle by Lemma \ref{lem:contains_cycle}.

So suppose that $d>1$.  }Let $\Omega$ be the orientable $d$-dimensional cycle in $\Delta$ given by our assumption and with $d$-faces $F_1,\ldots,F_m$.  By Lemma \ref{lem:orient_ddimcycle_is_dcycle} we know that the $d$-chain $c= F_1+\cdots+F_m$ is a homological $d$-cycle where $F_1,\ldots,F_m$ are oriented according to Definition \ref{def:orient_ddim_cycle}.  By assumption, $c$ is not a $d$-boundary and so $\tilde{H}_d(\Delta;k) \neq 0$.
\end{proof}

The converse of this theorem does not hold as can be seen from the example of the triangulation of the real projective plane.  As mentioned above this simplicial complex has non-zero $2$-dimensional homology over any field of characteristic 2, but this complex contains no orientable $2$-dimensional cycles.  The whole simplicial complex is, however, a non-orientable $2$-dimensional cycle.  Thus far we have been unable to find any counter-examples to the converse of Theorem \ref{thm:orientable_cycle_hom} in the case that the field in question has characteristic 0.

A triangulation of the mod $3$ Moore space $\Delta$, shown in Figure \ref{fig:Mod3Moore}, is another interesting counter-example to the converse of Theorem \ref{thm:orientable_cycle_hom}.  First, this is an example of a simplicial complex which has non-zero $2$-dimensional homology only over fields of characteristic $3$.  Second, and even more of note, is that $\Delta$ contains no orientable or non-orientable $2$-dimensional cycles.  Notice that $\Delta$ is the support complex of a minimal $2$-dimensional homological cycle over fields of characteristic $3$, in the sense that removing any of its facets leaves a simplicial complex with no $2$-dimensional homology. Also, in some sense, $\Delta$ is quite close to being a $2$-dimensional cycle.  All but three of its $1$-faces each lie in exactly two $2$-dimensional faces.  The remaining three $1$-faces lie in three $2$-dimensional faces each.  These $1$-faces are $\{x,y\}$, $\{x,z\}$, and $\{y,z\}$.  The face $\{x,y,z\}$ is not a face of $\Delta$, but by adding this face to $\Delta$ we obtain the $2$-dimensional cycle shown in Figure \ref{fig:Mod3Moore_plus}.

\begin{figure}[h!]
{\centering
    \includegraphics[height=1.7in]{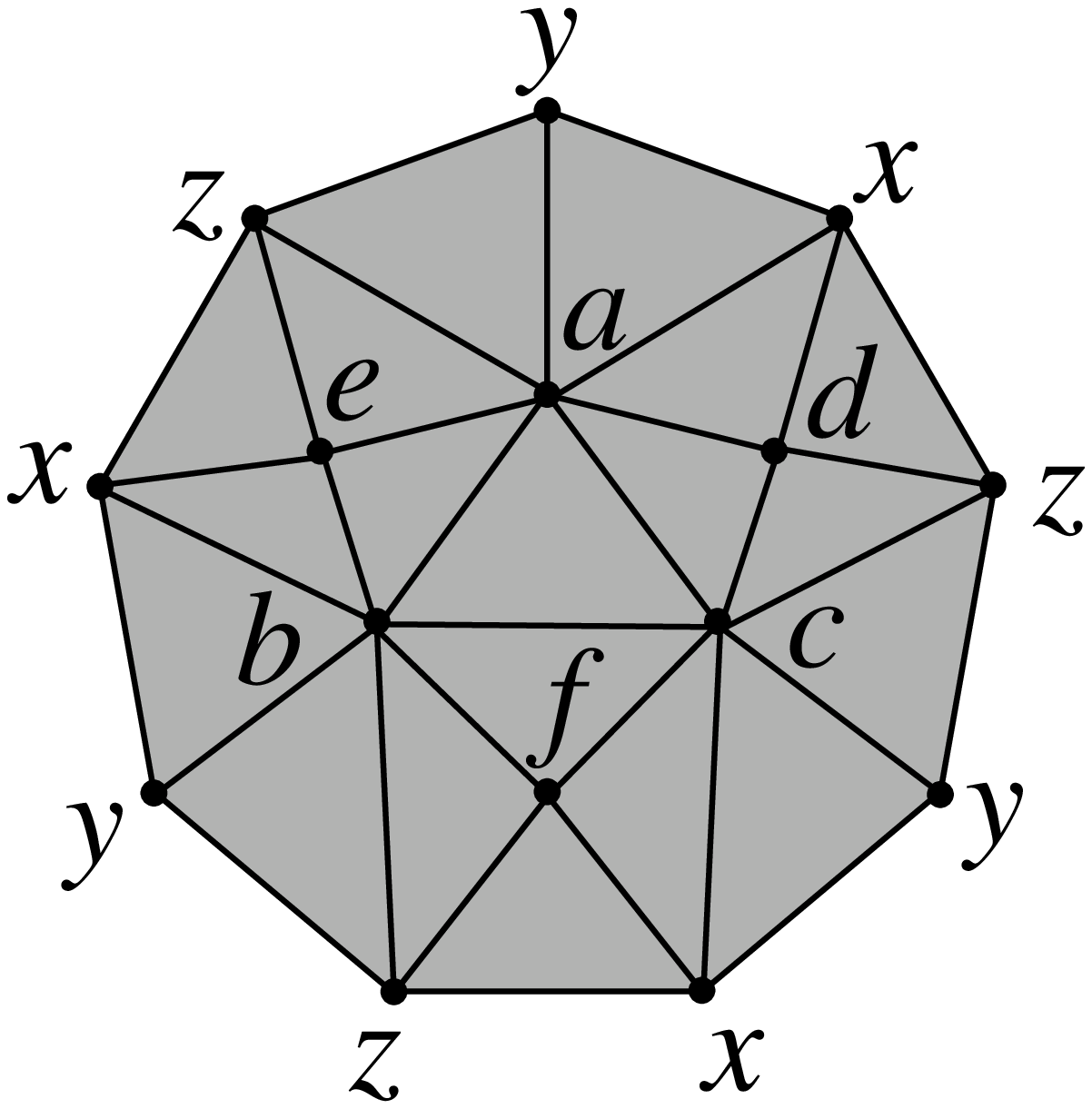}
    \caption {A triangulation of the mod $3$ Moore space.} \label{fig:Mod3Moore}
}
\end{figure}

\begin{figure}[h!]
{\centering
    \includegraphics[height=1.7in]{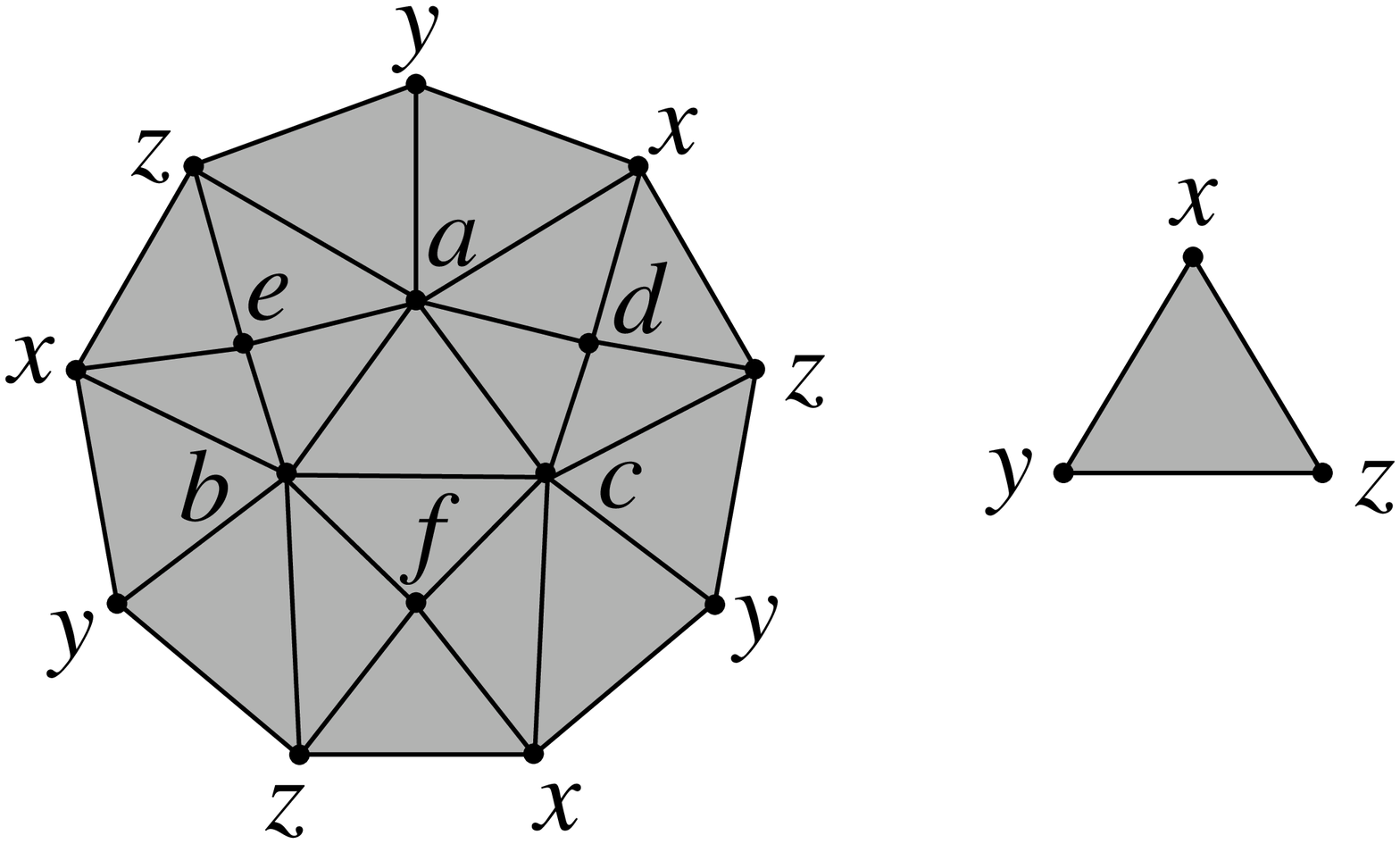}
    \caption {Modification of the triangulated mod $3$ Moore space.} \label{fig:Mod3Moore_plus}
}
\end{figure}

It is clear from the example of the mod $3$ Moore space that one may have non-zero homology over fields of finite characteristic without having $d$-dimensional cycles present.  However, thus far we have been unable to find any counter-examples to the converse of Theorem \ref{thm:orientable_cycle_hom} in the case that the field in question has characteristic 0.

The difficultly in trying to prove the converse of Theorem \ref{thm:orientable_cycle_hom} for the general case is in linking the algebraic notion of a homological $d$-cycle with general field coefficients to the combinatorics of its underlying support complex.  In particular, when the field coefficients of a $d$-chain are not all equal in absolute value it is difficult to find a combinatorial description of the role played by the coefficients.  Over $\Z$ it seems that many ``minimal'' homological $d$-cycles have coefficients which are constant up to absolute value.  Unfortunately the simplicial complex in Figure \ref{fig:Mod3Moore_plus} does not satisfy this property.  In this case, as a homological $2$-cycle over $\Z$, the coefficients of all facets but $\{x,y,z\}$ can be chosen to be $1$ in absolute value whereas the coefficient of $\{x,y,z\}$ must be $3$.


\bibliography{thesis_bib}

\begin{thebibliography}{10}

\bibitem{Berge89}
C.~Berge.
\newblock {\em Hypergraphs}.
\newblock Elsevier Science Publishers B. V., Amsterdam, The Netherlands, 1989.

\bibitem{Biggs93}
N.~Biggs.
\newblock {\em Algebraic Graph Theory}.
\newblock Cambridge University Press, 2nd edition, 1993.

\bibitem{Bjorn96}
A.~Bj\"{o}rner.
\newblock Handbook of combinatorics (vol. 2).
\newblock chapter Topological methods, pages 1819--1872. MIT Press, Cambridge,
  MA, USA, 1995.

\bibitem{CFS07}
M.~Caboara, S.~Faridi, and P.~Selinger.
\newblock Simplicial cycles and the computation of simplicial trees.
\newblock {\em J. Symbolic Comput.}, 42(1-2):74--88, 2007.

\bibitem{ConFar12}
E.~Connon and S.~Faridi.
\newblock Chorded complexes and a necessary condition for a monomial ideal to
  have a linear resolution.
\newblock {\em J. Comb. Theory Ser. A}, 120:1714--1731, 2013.

\bibitem{ConFar13}
E.~Connon and S.~Faridi.
\newblock A criterion for a monomial ideal to have a linear resolution in
  characteristic $2$.
\newblock arXiv:1306.2857, 2013.

\bibitem{Em10}
E.~Emtander.
\newblock A class of hypergraphs that generalizes chordal graphs.
\newblock {\em Math. Scand.}, 106(1):50--66, 2010.

\bibitem{Far04}
S.~Faridi.
\newblock Simplicial trees are sequentially {C}ohen-{M}acaulay.
\newblock {\em J. Pure Appl. Algebra}, 190(1--3):121--136, 2004.

\bibitem{Fog88}
A.~L. Fogelsanger.
\newblock {\em The Generic Rigidity of Minimal Cycles}.
\newblock PhD thesis, Cornell University, Ithaca, New York, May 1988.

\bibitem{Fr85}
R.~Fr\"oberg.
\newblock Rings with monomial relations having linear resolutions.
\newblock {\em J. Pure Appl. Algebra}, 38:235--241, 1985.

\bibitem{Fr90}
R.~Fr\"oberg.
\newblock On {S}tanley-{R}eisner rings.
\newblock In {\em Topics in Algebra, \emph{Part II}}, volume~26, pages 57--70,
  Warsaw, 1990. Banach Center Publ., PWN.

\bibitem{HaVT08}
H.~T. H\`a and A.~Van~Tuyl.
\newblock Monomial ideals, edge ideals of hypergraphs, and their graded {B}etti
  numbers.
\newblock {\em J. Algebraic Combin.}, 27(2):215--245, 2008.

\bibitem{Hatch02}
A.~Hatcher.
\newblock {\em Algebraic Topology}.
\newblock Cambridge University Press, Cambridge, 2002.

\bibitem{HHZ06}
J.~Herzog, T.~Hibi, and X.~Zheng.
\newblock {C}ohen-{M}acaulay chordal graphs.
\newblock {\em J. Comb. Theory Ser. A}, 113(5):911--916, 2006.

\bibitem{Kal84}
G.~Kalai.
\newblock $f$-{V}ectors of acyclic complexes.
\newblock {\em Discrete Math.}, 55:97--99, 1984.

\bibitem{MYZ12}
M.~Morales, A.~A. Yazdan~Pour, and R.~Zaare-Nahandi.
\newblock Regularity and free resolution of ideals which are minimal to
  $d$-linearity.
\newblock 2012.
\newblock arXiv:1207.1790v1.

\bibitem{Munk84}
J.~R. Munkres.
\newblock {\em Elements of Algebraic Topology}.
\newblock The Benjamin/Cummings Publishing Company, Inc., Menlo Park,
  California, 1984.

\bibitem{Reis76}
G.~A. Reisner.
\newblock Cohen-{M}acaulay quotients of polynomial rings.
\newblock {\em Advances in Math.}, 21:30--49, 1976.

\bibitem{Stan93}
R.~P. Stanley.
\newblock A combinatorial decomposition of acyclic simplicial complexes.
\newblock {\em Discrete Math.}, 120:175--182, 1993.

\bibitem{Vill90}
R.~H. Villarreal.
\newblock Cohen-{M}acaulay graphs.
\newblock {\em Manuscripta Math.}, 66:277--293, 1990.

\bibitem{West96}
D.~B. West.
\newblock {\em Introduction to Graph Theory}.
\newblock Prentice-Hall, Upper Sadlle River, NJ, 1996.

\bibitem{Wood11}
R.~Woodroofe.
\newblock Chordal and sequentially {C}ohen-{M}acaulay clutters.
\newblock {\em Electron. J. Combin.}, 18(1), 2011.
\newblock Paper 208.

\end{thebibliography}

\end{document}